\documentclass[12pt,titlepage]{amsart}
\usepackage{setspace}
\usepackage{amsmath,amsthm,amssymb,amsfonts}

\usepackage{hyperref}
\hypersetup{colorlinks=true}
\numberwithin{equation}{section}
\makeindex

\newcommand\xappa\kappa
\newcommand\yota\iota
\newcounter{consta}

\newcounter{constb}

\newcounter{constc}[section]

\DeclareFontFamily{OT1}{rsfs}{}
\DeclareFontShape{OT1}{rsfs}{n}{it}{<-> rsfs10}{}
\DeclareMathAlphabet{\mathscr}{OT1}{rsfs}{n}{it}
\swapnumbers
\newtheorem{thm}{Theorem}[section]
\newtheorem{lem}[thm]{Lemma}

\newtheorem{prop}[thm]{Proposition}
\newtheorem*{lem*}{Lemma}
\newtheorem*{thm*}{Theorem}
\newtheorem*{conj*}{Conjecture}
\newtheorem*{prop*}{Proposition}
\newtheorem{defn*}{Definition}

\newtheorem{cor}[thm]{Corollary}

\theoremstyle{definition}
\newtheorem{defn}[thm]{Definition}
\theoremstyle{remark}
\newtheorem{rem}[thm]{Remark}
\newtheorem{obs}[thm]{Observation}
\newtheorem{question}[thm]{Question}
\newtheorem*{obs*}{Observation}
\newtheorem*{rem*}{Remark}
\theoremstyle{definition}\newtheorem*{acknowledgments}{Acknowledgments}

\begin{document}
\title[On mixing and sparse ergodic theorems]{On mixing and sparse ergodic theorems}
\author{Asaf Katz}
\address{Einstein Institute of Mathematics, The Hebrew University of Jerusalem, Jerusalem, 91904, Israel.}
\curraddr{Depaterment of Matheamtics, University of Michigan, Ann Arbor, MI, 48109, USA.}
\email{asaf.katz@gmail.com}

%\thanks{The research was supported by ERC grant (AdG Grant 267259).}
\date{}
%\pagenumbering{gobble}

\begin{abstract} We consider Bourgain's ergodic theorem regarding arithmetic averages in the cases where quantitative mixing is present in the dynamical system.
Focusing on the case of the horocyclic flow, those estimates allow us to bound from above the Hausdorff dimension of the exceptional set, providing evidence towards conjectures by Margulis, Shah and Sarnak regarding equidistribution of arithmetic averages in homogeneous spaces. We also prove the existence of a uniform upper bound for the Hausdorff dimension of the exceptional set which is independent of the spectral gap.
\end{abstract}
\maketitle

\section{Introduction}
In a seminal paper \cite{bourgain1989pointwise}, J. Bourgain proved a pointwise ergodic theorem for arithmetic averages, solving an open problem due to H. Furstenberg. We refer the interested reader to Bourgain's exposition in \cite{bourgain1988approach} and the comprehensive survey article~\cite{rosenblatt1995pointwise}.

In this paper, we study Bourgain's theorem for dynamical systems for which quantitative mixing estimates hold and in particular in the context of unipotent flows on homogeneous spaces.

It has been conjectured by N. Shah~\cite{shah1994limit}, in the case of the horocyclic flow on homogeneous spaces of $SL_{2}(\mathbb{R})$, and in a general form by G. Margulis~\cite{margulis2000problems} and in another form by P. Sarnak~\cite{sarnak2010mobius}, that those arithmetic averages should converge for \emph{every individual point}, for every homogeneous space $G/\Gamma$ where $\Gamma$ is a lattice in $G$. 

The analogous settings for \emph{continuous time} flows is addressed by equidistribution theorems for unipotent flows such as the Dani-Smillie theorem (\cite[Theorem~$1$]{dani1984}, Ratner's equidistribution theorem, Shah's equidistribution theorem (\cite[Theorem $1.1$, Corollary $1.1$]{shah1994limit}) and various effective improvements of those theorems.

We provide evidence towards those conjectures in the form of limiting the Hausdorff dimension of the exceptional set.
The main ingredients in our proof are \emph{quantitative mixing estimates} and the \emph{polynomial rate of divergence for the horocyclic flow}.
This approach is different than the approach which has been used towards those conjectures by A. Venkatesh \cite{venkatesh2010sparse}, which is based upon proving a ''large level of distribution`` in a suitable formulation of the quantitative pointwise ergodic theorem.

The analogous situation for nilflows has been proven by Leibman in \cite{leibman2005pointwise} and in quantitative form by Green-Tao in \cite{green2012quantitative}.

\subsection*{Statement of the results}

Our main result is:
\begin{thm}\label{thm:hdim-estimate-absolute}
Let $G$ be a simple Lie group, $U$ a one-parameter unipotent subgroup $U=\{u_{t}\}_{t\in \mathbb{R}} \leq G$ and $p(x)$ be a non-constant polynomial with integer coefficients. There exists a constant $\sigma=\sigma(G,p)>0$ such that for any lattice $\Gamma \leq G$, when considering the $U$-flow on the homogeneous space $X=G/\Gamma$, the following estimate holds 
\begin{equation*}
    dim_{H}\left\{x\in X \mid \left\{u_{p(n)}.x\right\}_{n\in \mathbb{N}} \text{ does not equidistribute} \right\} \leq dim(G)-\sigma,
\end{equation*}
where $\dim_{H}$ stands for Hausdorff dimension.
\end{thm}
The proof is based on quantitative mixing estimates, geometrical analysis of the polynomial divergence behavior of the unipotent flow, Ratner's measure classification theorem in order to control associated continuous time averages and upon number-theoretical bounds for moments of exponential sums.\\

Building towards our main theorem, we deduce the following theorems, which are of independent interest.
In order to state our results, we introduce the following definitions: 
\begin{defn} We say that a strictly increasing sequence of integers $\{a_{n}\}_{n\in \mathbb{N}} \subset~\mathbb{N}$ \emph{grows polynomially} if there exists $d\in \mathbb{N}$, and some $C>0$ such that $a_{n} \leq Cn^{d}$. In this case, we will say that $\{a_n\}$ grows with rate $d$.
\end{defn}
\begin{defn}\label{defn:poly-mixing} Let $\left(X,\mathcal{B},\mu,T \right)$ be a measure preserving system.
Denote $L_{0}^{2}(X,\mu)$ the space of square-integrable functions over $X$ with vanishing integral.
Fix a function $f\in L^{2}_{0}\left(X,\mu\right)$.
We say that $\left(X,T\right)$ is \emph{polynomially mixing} for $f$ with rate $\alpha$ if there exists $\alpha>0$ such that $\left|\left<T^{n}f,f \right>\right| \leq C n^{-\alpha}$ for all $n\in\mathbb{N}$ and for some constant $C=~C(f)$, where $\left< f,g\right>=~\int_{X}f(x)\overline{g(x)}d\mu(x)$.
If $\int_{X}fd\mu(x)\neq 0$, we say that the function  $f$ is polynomially mixing if its projection into $L^{2}_{0}(X,\mathbb{\mu})$, namely $f(x)-~\int_{X}fd\mu$ is polynomially mixing.
\end{defn}

\begin{defn}
For a function $f:X\to \mathbb{C}$ where $X=(X,d)$ a metric space, we define the Lipschitz norm as follows:
\begin{equation}\label{eq:Lip-norm}
    \lVert f \rVert_{Lip} = \lVert f \rVert_{\infty} + \sup_{x\neq y \in X} \left\lvert \frac{f(x)-f(y)}{d_{X}(x,y)} \right\rvert.
\end{equation}
In the case of $X=G/\Gamma$ we endow this space with a metric $d_{X}$ induced from a right-invariant metric on the group $G$.
\end{defn}

\begin{thm}\label{thm:hdim-estimate} Let $G=SL_{2}(\mathbb{R})$, $\Gamma \leq G$ a lattice, $\mu$ the unique $G$-invariant probability measure on $X=G/\Gamma$ and $\{u_{t}\}$ the horocyclic flow defined on $X$. 
If $f$ is a bounded Lipschitz  function satisfying $\int_{X}fd\mu=~0$ for which $(X,u_{t})$ is mixing with polynomial rate $\alpha$ for some $\alpha>0$, then for any sequence $\{a_{n}\}_{n\in \mathbb{N}} \subset \mathbb{N}$ which grows polynomially with rate $d$, the following estimate holds:
\begin{align*}
\dim_{H}\left\{x\in X \middle| \overline{\lim} \left\lvert\frac{1}{N}\sum_{n=1}^{N}f(u_{a_{n}}.x)\right\rvert>0 \right\} \leq 3-\frac{\alpha'}{d},
\end{align*} 
for some positive $\alpha'=\alpha'(d,\alpha,\Gamma)$.
\end{thm}

As a corollary, combining the above theorem with known bounds for the decay of matrix coefficients and an approximation argument, we get:
\begin{cor}\label{cor:decay-estimate}
Let $G=SL_{2}(\mathbb{R})$, $\Gamma \leq G$ a lattice, $\mu$ the unique $G$-invariant probability measure on $X=G/\Gamma$ and $\{u_{t}\}$ the horocyclic flow defined on $X$. 
Fix a non-constant polynomial $p\in \mathbb{Z}[x]$ with $\deg(p)=d$.
There exists a number $s=s(\Gamma,p)>0$ such that for every bounded Lipschitz function $f$ the following estimate holds:  
\begin{align*}
\dim_{H}\left\{x\in X \middle| \overline{\lim} \left\lvert\frac{1}{N}\sum_{i=1}^{N}f(u_{p(i)}.x)-\int_{X}fd\mu\right\rvert>0 \right\} \leq 3-\frac{s}{d}.
\end{align*} 
\end{cor}
The number $s(\Gamma,p)$ is related to the \emph{spectral gap} of $G/\Gamma$. This number can be explicitly calculated. Let $\lambda_{1}$ be the smallest non-zero eigenvalue of the Laplacian over the hyperbolic surface $K\backslash G/\Gamma$. We parametrize $\lambda_{1}$ as $\lambda_{1}=s_{1}\cdot(1-s_{1})$ with $\Re\left(s_{1}\right)\leq \frac{1}{2}$, and pick $s=\min\{(1/2),~d~\cdot~\Re(s_1)\}$. 
Theorem~\ref{thm:hdim-estimate-absolute} strengthens this result to be free of spectral gap assumptions, but is dependent on particular arithmetic sampling sequences in order to achieve cancellation of certain exponential sums.

\subsection*{Organization of the paper}
The article is divided into five sections. 

In $\S\ref{sec:ergodic-theorem}$, we prove, in a rather abstract setting, a weak variant of Bourgain's sparse ergodic theorem in the presence of quantitative mixing estimates. The main technique used in the proof is the utilization of the quantitative mixing directly in $L^{2}(X)$, unlike Bourgain's method which involves passing to $\ell^{2}(\mathbb{Z})$ by means of Calderon transference.

In $\S\ref{sec:h-dim}$, we focus on the case of the horocyclic flow on homogeneous spaces of $SL_{2}(\mathbb{R})$, and arithmetic averages along the horocyclic flow, proving Theorem~\ref{thm:hdim-estimate}.
We are able to deduce the estimates about the Hausdorff dimension of the exceptional set via a packing argument.

In $\S\ref{sec:s-gap}$, we show an explicit method to bound from above the exceptional set uniformly, without dependence on the spectral gap of the given lattice in $PSL_{2}$, which improves upon the results of $\S$~\ref{sec:h-dim} in the case of very small spectral gap. The proof involves a careful study of the action of the associated averaging operator on complementary series representations and the Dani-Smillie theorem.

In $\S\ref{sec:higher-dimensions}$ we extend our results for the case of general simple Lie groups and one-parameter unipotent flows, by using Ratner's measure-classification theorem and equidistribution theorem.

\begin{acknowledgments}
The results of this paper were obtained as part of the author's PhD thesis at the Hebrew University
of Jerusalem under the guidance of Prof. Elon Lindenstrauss. The author also wishes to thank Prof. Tamar Ziegler, Prof. Shahar Mozes and Prof. Nimish Shah for useful conversations during the research. Part of the research was done while the author was staying at MSRI during the program ''Geometric and Arithmetic Aspects of Homogeneous Dynamics'', the author wishes to thank MSRI and the program organizers for their hospitality.
The research was supported by ERC grant (AdG Grant 267259).
The author would like to the thank the referee for doing an excellent job and improving the presentation of the argument.

\end{acknowledgments}

\section{Proof of a sparse ergodic theorem}\label{sec:ergodic-theorem}
We begin by proving a variant of Bourgain's theorem where quantitative mixing is present.
Although such a result is substantially weaker than Bourgain's (i.e. even for the Bernoulli shift, not all $L^{2}$-functions satisfy the polynomial mixing requirement), in practice for many interesting applications (especially in homogeneous dynamics) one is able to verify such conditions for the functions in question.
Moreover, the spectral estimate achieved in the course of the proof will play a key role in subsequent sections, hence we provide the details for the sake of completeness.

For the rest of this section, fix $(X,\mathcal{B},\mu,T)$ a measure preserving system, $f\in~L^2(X,\mu)$ a bounded function with $\int_{X}fd\mu=0$, $\{n_i\}_i\in\mathbb{N}$ a sequence of integers which grows polynomially with rate $d>0$, and the following operator
\begin{equation}\label{eq:def-avg-op-aribtrary}
    A_{N}f(x)=\frac{1}{N}\sum_{i=0}^{N-1}f(T^{n_{i}}x).
\end{equation}

We being with the following lemma.
\begin{lem}\label{lem:quant-alpha-mixing} Assume that $(X,T)$ is polynomially mixing for $f$ with rate $\alpha>0$, and that $\int_{X}fd\mu=0$, then for $\alpha'=\frac{1}{2}\min\{1,\alpha\}$ we have  that 
\begin{equation*}
    \left\| A_{N}f\right\|_{L^{2}(X,\mu)} \leq C_{\text{mix}}\cdot N^{-\alpha'}\|f\|_{L^{2}(X,\mu)},
\end{equation*}
for some constant $C_{\text{mix}}=C_{\text{mix}}(f)$.
\end{lem}
\begin{proof}
By explicit computation - 
\begin{align*}
\left<A_{N}f,A_{N}f\right> &=  \left<\frac{1}{N}\sum_{i=1}^{N}T^{n_{i}}.f,\frac{1}{N}\sum_{j=1}^{N}T^{n_{j}}.f\right> \\
								&=  \frac{1}{N^2}\left<\sum_{1\leq i,j\leq N} T^{n_{i}-n_{j}}.f,f\right> \\
								&=  \frac{1}{N^2}\sum_{k=-n_{N}+1}^{n_{N}-1}d_{N}(k)\left<T^{k}.f,f\right>, 
\end{align*}
where we define $d_{N}(k)$ as follows: $$d_{N}(k)=\lvert\{(i,j) \mid k=n_{i}-n_{j}, 1\leq i,j\leq N \}\rvert.$$
As we have $d_{N}(k) \leq N-1$ for any $-n_{N}+1\leq k \leq n_{N}-1$, and using the bounds for  $\lvert\left<T^{n}f,f\right>\rvert$ that we have due to polynomial mixing, we deduce that $\left<A_{N}f,A_{N}f\right> \leq \frac{\|f\|^{2}_{L^{2}(X,\mu)}}{N}+\frac{1}{N}\sum_{k=-N+1,\ k\neq 0}^{N-1}\lvert\left<T^{k}.f,f\right>\rvert$, and using the explicit mixing rate we can conclude - 
\begin{align*} 
\left<A_{N}f,A_{N}f\right> &\leq \frac{\|f\|^{2}_{L^{2}(X,\mu)}}{N} + \frac{1}{N}\sum_{k=-N+1,\ k\neq 0}^{N-1}C'\cdot \lvert k\rvert^{-\alpha}\|f\|^{2}_{L^{2}(X,\mu)} \\
&\leq \frac{\|f\|^{2}_{L^{2}(X,\mu)}}{N}+ 2C'\cdot N^{-\alpha}\|f\|^{2}_{L^{2}(X,\mu)},
\end{align*}
for some $C'$ which depends on $f$ by means of its polynomial mixing rate (c.f. Definition~\ref{defn:poly-mixing}).
Hence the lemma follows by choosing $C_{\text{mix}}~=~2\cdot~C'+1$.
\end{proof}

We use the following bootstrapping lemma, used by Bourgain, which allows us to bootstrap convergence along ''slowly lacunary'' subsequence of the averaging operators $\{A_N\}$ to a convergence of the full sequence $\{A_N\}$, the proof is included for the sake of completeness.

\begin{lem}\label{lem:bootstrap}
For a bounded function $f$, the sequence $\{A_{N}f(x)\}_{N\in \mathbb{N}}$ converges (at the point $x$) if and only if for every $\varepsilon>0$ the sub-sequence $\{A_{[(1+\varepsilon)^{N}]}f(x)\}_{N\in \mathbb{N}}$ converges (for the same point $x$).
\end{lem}
\begin{proof} Fix $\varepsilon > 0$, note that for every $N \in \mathbb{N}$ there exists an integer of the form $[(1+\varepsilon)^{m}]$ between 
$N$ and $N(1+\varepsilon)$ for some integer $m$.
By a direct comparison -
\begin{align*}
\left\lvert A_{N}f(x)-A_{[(1+\varepsilon)^{m}]}f(x) \right\rvert &\leq \frac{[(1+\varepsilon)^{m}]-N}{[(1+\varepsilon)^{m}]}\|f\|_{\infty} + \frac{[(1+\varepsilon)^{m}]-N}{[(1+\varepsilon)^{m}]}A_{N}|f(x)| \\
&\leq \frac{2\varepsilon N}{[(1+\varepsilon)^{m}]}\cdot \|f\|_{\infty} \\
&\leq 2\cdot\varepsilon \cdot \|f\|_{\infty}, 
\end{align*}
and the proof follows from the pointwise convergence of the sub-sequence $\left\{A_{[(1+\varepsilon)^{m}]}f(x)\right\}$.
\end{proof}
\begin{rem}
It is enough to verify Lemma~\ref{lem:bootstrap} for a monotonically decreasing sequence of numbers $\left\{\varepsilon_{i}\right\}$ such that $\lim_{i\to \infty}\varepsilon_{i}=0$.
\end{rem}

Using the quantitative estimates, we are able to prove the following analogue of Bourgain's sparse ergodic theorem.
\begin{thm}\label{thm:sparse-ergodic-thm}\hypertarget{thm:sparse-ergodic-thm} Let $\left(X,\mathcal{B},\mu,T \right)$ be a measure preserving system,\\ $f~\in~ L^{2}\cap~L^{\infty}\left(X,\mu\right)$ and fix a monotone infinite sequence $\{n_{i}\}_{i=1}^{\infty}\subset\mathbb{N}$. If there exists some $\alpha>0$ such that $\left(X,T\right)$ \emph{is mixing with polynomial rate $\alpha$} for $f$ then for $\left[\mu\right]$-almost-every point $x\in X$, the averages $A_{N}f(x):=\frac{1}{N}\sum_{i=1}^{N}f(T^{n_{i}}x)$ converge to $\int_{X} f(x)d\mu(x)$.
\end{thm}

\begin{proof}
First we can assume that the integral of $f$ equals $0$, if not we replace $f$ by $f-\int_{X}fd\mu$. Fix some $\gamma>0$, and define the exceptional sets for decay at rate $\gamma$ to be $E^{\gamma}_{N}=\{x\in X \mid \lvert A_{N}f(x)\rvert >N^{-\gamma}\}$. By Lemma~\ref{lem:bootstrap} (and the remark following it), it is enough to consider convergence along the lacunary sub-sequences $\{A_{[(1+\varepsilon)^{m}]}f\}$ of the averages $\{A_{N}f\}$, for some countable set of positive numbers $\varepsilon$, decreasing to $0$.
Using Chebyshev's inequality we estimate the measures of the exceptional sets as $$\mu\left(E^{\gamma}_{[(1+\varepsilon)^{m}]}\right)\leq [(1+\varepsilon)^{m}]^{2\gamma}\|A_{[(1+\varepsilon)^{m}]}f\|_{L^{2}(X,\mu)}^{2}.$$
By Lemma~\ref{lem:quant-alpha-mixing} we get $\|A_{[(1+\varepsilon)^{m}]}f\|_{L^{2}(X,\mu)}^{2} \ll_{f} [(1+\varepsilon)^{m}]^{-2\alpha'}$, \\
where the implied constant depends on $f$ and the given sampling sequence, as in the proof of this lemma.
As for all $0<\gamma<\alpha'$ we have that $\sum_{m=1}^{\infty} [(1+\varepsilon)^{m}]^{2\gamma-2\alpha'} < \infty$, by the Borel-Cantelli lemma, $\mu\left(\limsup E^\gamma_{[(1+\varepsilon)^{m}]}\right)=~0$, concluding convergence along the lacunary subsequence.
\end{proof}

\section{Bounding the exceptional set}\label{sec:h-dim}
In \cite{shah1994limit}, N. Shah has asked the following question, which is related to a previous question by Margulis:
\begin{question}
Let $G=SL_{2}(\mathbb{R})$ ,$\Gamma \leq G$ be a lattice and $X$ be the homogeneous space $X=G/\Gamma$. Let $U=\{u_{t}\}$ be the upper unipotent group, namely $u_{t}=\left(\begin{smallmatrix} 1 & t \\ 0 & 1\end{smallmatrix} \right)$. Given $f\in C_{c}(X)$, is it true that the horocyclic  averages along the squares, $A_{N}f(x)=\frac{1}{N}\sum_{n=1}^{N}f(u_{n^2}.x)$, converges everywhere?
\end{question}
The almost-surely result follows directly from Bourgain's theorem, and the continuous time analogue of the question was proven (as part of a much more general theorem) by Shah in \cite{shah1994limit} using measure-classification techniques.

The current approach towards this question, pioneered by A. Venkatesh in \cite{venkatesh2010sparse}, asks for a quantitative pointwise ergodic theorem for the continuous time flow (which has been studied by numerous authors, see \cite{burger1990horocycle}, \cite{flaminio2003invariant}, \cite{strombergsson2013deviation} and \cite{sarnak2015horocycle}) and then deduces a quantitative pointwise Wiener-Wintner ergodic theorem, namely quantify the disjointness of the horocyclic flow from a Kronecker system (see similar results in \cite{tanis2015uniform}, \cite{flaminio2015effective} and \cite{zheng2015sparse}).
One then approximates the sampling sequence as an arithmetic progression and using the disjointness one basically reduces the question to a question about ''level of distribution'' achieved in the quantitative pointwise ergodic theorem.
Unfortunately, the current techniques involved in deducing a quantitative pointwise ergodic theorem are not strong enough, even under the condition of the Selberg-Ramanujan conjecture, in order to prove the result regarding average along the squares.
Moreover, as the approximation is done by linear functions, this method has a natural threshold at the squares and can not be applicable to cubes or other higher powers.
Another problem with the current approach arise in the case where the lattice $\Gamma$ is non-uniform, where there is no proper non-divergence argument available for the case of the squares (for the continuous time flow, the Dani-Margulis lemma provides non-divergence of general continuous-time ''polynomial orbits'').

It is worth mentioning an interesting work done by P. Sarnak and A. Ubis in \cite{sarnak2015horocycle}, where they have carefully deduced a quantitative pointwise ergodic theorem in the case of congruence lattices of $SL_{2}(\mathbb{Z})$, and a result towards the horocyclic flow average along the primes. 

In the primes case, one does not need to approximate by arithmetic progression, but by standard sieving arguments, one is led to using Vinogradov's summation technique and the questions regarding ''level of distribution'' for the horocyclic flow arise again.

Our approach is different, more modest in its aim, providing a direct packing argument which bounds the dimension of the exceptional set for those questions. 
The method is flexible and can be adopted to a more general situation of sparse averages which are taken on maximal horospherical subgroups.

We recall the definition of Hausdorff dimension for a metric space $Y$.

\begin{defn}
For $D\geq 0$ the \emph{D-dimensional Hausdorff measure} of a set $B\subset Y$ is defined by 
\begin{equation*}
    \mathcal{H}^{D}(B)=\lim_{\epsilon\to 0}\inf_{C_{\epsilon}}\sum_{i}\left(\text{diam}\left(C_{i}\right)\right)^{D},
\end{equation*}
where $C_\epsilon=\left\{C_{1},C_{2},\ldots\right\}$ is any countable cover of $B$ with sets $C_{i}$ of diameter $\text{diam}\left(C_{i}\right)$ less than $\epsilon$.
The \emph{Hausdorff dimension} of $B$ is defined by 
\begin{equation}\label{eq:hdim-def}
    \dim_{H}(B)=\inf\left\{D \mid \mathcal{H}^{D}(B)=0 \right\}=\sup\left\{ D \mid \mathcal{H}^{D}(B)=\infty\right\}.
\end{equation}
\end{defn}
\begin{defn}\label{defn:seperated-set}
For every $\delta>0$ a set $F\subset B$ is called \emph{$\delta$-separated} if $d_{Y}(x,y)\geq \delta$ for every two distinct points $x,y\in F$. We denote the cardinality of the largest $\delta$-separated subset of $B$ by $N(B,\delta)$.
\end{defn}

We are interested in estimating the Hausdorff dimension of a set $A$, which satisfies $$A\subset \bigcap_{k\geq 1}\bigcup_{i\geq k}A_{i}.$$
The main technique we will use in order to estimate the Hausdorff dimension will be as follows:
We choose some $\delta_{1},\delta_{2},\ldots >0$ such that $\delta_{i}\to 0$ as $i\to \infty$.
For each $i$, we let $F_{i}$ be a maximal $\delta_i$-separated subset of $A_{i}$.
Then we have that $\lvert F_{i}\rvert \leq N(A_{i},\delta_{i})$.
Also we have that $A_{i}\subset~\cup_{x\in F_{i}}B_{\delta_i}(x)$, since if there exists a point $y\in A_{i}$ outside of $\cup_{x\in F_i}B_{\delta_i}(x)$, then $F_{i}\cup\{y\}$ is a $\delta_{i}$-separated set strictly containing $F_{i}$.
Considering the cover of $A_{i}$ formed by the $\delta_{i}$-balls, one deduces that
$$ \sum_{x\in F_{i}} (\text{diam}(B_{\delta_i}))^{D} \leq \lvert F_{i} \rvert \cdot (2\cdot \delta_{i})^{D} \leq 2^{D}\cdot N(A_{i},\delta_{i})\cdot \delta_{i}^{D}.  $$
Moreover, by aggregating these coverings for any $i$ one gets
\begin{equation*}
    \bigcup_{i\geq k}A_{i} \subset \bigcup_{i\geq k}\bigcup_{x\in F_i}B_{x}(\delta_i).
\end{equation*}
Now assume 
\begin{equation}\label{eq:dim-bound-seperation}
    \sum_{i\geq 1}N(A_{i},\delta_i)\cdot\delta_{i}^{D}<\infty.
\end{equation}
For any $\varepsilon>0$ we may take $k_0$ large such that $\sum_{i\geq k_0}N(A_{i},\delta_i)\cdot(2\cdot~\delta_{i})^{D}<~\varepsilon$ and $2\cdot\delta_{i}<\varepsilon$ for all $i\geq k_0$.
It follows that the sets $\left\{B_{\delta_i}(x) \mid i\geq k_0, \ x\in F_{i} \right\}$ for any $i\geq k_0$ form a cover of $\cap_{k\geq 1}\cup_{i\geq k}A_{i}$ with sets which all have diameter smaller than $\varepsilon$.
For this cover we have 
\begin{align*}
    \sum_{i\geq k_0}\sum_{x\in F_i}(\text{diam} B_{\delta_i}(x))^{D} &\leq \sum_{i\geq k_0}N(A_{i},\delta_i)\cdot (2\delta_{i})^{D} \\
    &< \varepsilon.
\end{align*}
Using the properties of the Hausdorff dimension, the existence of such cover for any $\varepsilon>0$ implies that $\mathcal{H}^{D}(\cap_{k\geq 1}\cup_{i\geq k}A_{i})=0$, hence afortiori $\mathcal{H}^{D}(A)=0$.

From now on, we fix $G=SL_{2}(\mathbb{R})$, $\Gamma \leq G$ a lattice, $X=G/\Gamma$ and $\mu$ denotes the unique $G$-invariant probability measure on $X$, normalized so that $\mu(X)=1$.
Let $K\leq G$ stands for $SO(2)$ - the standard  maximal compact subgroup of $G$.
We denote by $d_{G}$ a right-invariant metric on the Lie group $G$. This metric descends to a metric $d_{X}$ on the homogeneous space $X=G/\Gamma$.
We begin with an auxiliary approximation lemma.
\begin{defn}\label{defn:K-finite}
A vector $v\in V$ where $V$ is some unitary $G$-representation is called \emph{K-finite} if its $K$-span $$\left<K.v\right> = \overline{\text{span}\left\{k.v \mid k\in K \right\}}\leq V $$ is finite dimensional.
\end{defn}
\begin{lem}\label{lem:lip-approx}
Let $f\in C_{c}(X)$ be a Lipschitz function relative to a metric $d_{X}$ on $X$. For every $\varepsilon>0$ there exists a \emph{K-finite} function $\tilde{f}$ for which 
\begin{equation*}
   \|f-\tilde{f}\|_{\infty} \leq \varepsilon,
\end{equation*}
moreover we have that
\begin{equation*}
    \|\tilde{f}\|_{L^{2}(\mu)} \leq~ \left\|f\right\|_{L^{2}(\mu)}.
\end{equation*}

\end{lem}
\begin{proof}
Denote for every positive integer $L$ the \emph{Fejer kernel} as $$F_{L}(k)=\sum_{\lvert j \rvert\ \leq L}\left(1-\frac{\lvert j \rvert}{L} \right)e_{j}(k),$$ where $e_{j}(k)=e^{2\pi i j k}$. Define the following function:
\begin{equation*}
g^{\left[-L,L\right]}(x)=\int_{K}f(k\cdot x)F_{L}(k)dk,
\end{equation*}
 where the integration is done with respect to the Haar measure on $K$.
 
For any given $x\in X$, $g^{\left[-L,L\right]}(x)$ converges to $f(x)$, by Fejer's theorem~\cite[Theorem~$I.3.1$]{katznelson}.
Readily $g^{\left[-L,L\right]}$ is a $K$-finite function with $\dim\left<K\cdot g^{\left[-L,L\right]}\right> \leq 2L+1$, as can be seen by extending $f(kx)$ to a Fourier series, and utilizing orthogonality of $K$-characters using the dominated convergence theorem. 
Moreover, we have the following estimate for the error of the Fejer kernel (c.f.~\cite[Exercise~$I.3.1$]{katznelson}), as $f$ is a Lipschitz function:
\begin{equation*}
    \left\lvert g^{\left[-L,L\right]}(x)-f(x)  \right\rvert \ll_{f}  \frac{\log(L)}{L},
\end{equation*}
where the dependence is by means of the Lipschitz norm of $f$, 
$$ \lVert f \rVert_{Lip}=\lVert f \rVert_{\infty} + \sup_{x\neq y}\left\lvert \frac{f(x)-f(y)}{d_{X}(x,y)} \right\rvert.  $$
Choosing $L$ large enough and defining $\tilde{f}=g^{\left[-L,L\right]}$ we deduce the theorem.
Furthermore, we see by Parseval's identity that $\|\tilde{f}\|_{L^{2}(\mu)} \leq~ \left\|f\right\|_{L^{2}(\mu)}$, verifying the second assertion.
\end{proof}

\begin{defn}\label{defn:good-point} 
Fix an increasing sequence of natural numbers $\left\{a_{n} \right\}$.
Define the sequence of averaging operators $A_{N}:C_{c}(X)\to\mathbb{C}$ relative to the sequence $\left\{a_n\right\}$ as
\begin{equation*}
    A_{N}f(x) = \frac{1}{N}\sum_{n=1}^{N}f(T^{a_{n}}.x),
\end{equation*}
for any $f\in C_{c}(X)\to \mathbb{R}$.
Fix a bounded Lipschitz function $f:X\to \mathbb{C}$ with $\int_{X}fd\mu =0$.
Let $C>0$ be a fixed constant, which we allow to depend on the sampling sequence $\left\{a_n \right\}$.
We say that a point $x\in X$ is an \emph{$(N,\gamma)$-Good point} with a bound $C$ for $N\in\mathbb{N}$ and $\gamma>0$, if $$\left\lvert A_{N}f(x)\right\rvert\leq C\cdot N^{-\gamma} \cdot \lVert f \rVert_{Lip},$$ namely the average $A_{N}f(x)$ is bounded by $N^{-\gamma}$, up to the fixed factor $C$ and the Lipschitz norm.
We remind that the reader that by our definition of the Lipschitz norm~\eqref{eq:Lip-norm}, the Lipschitz norm dominates the $L^{2}$ norm.

We define the set of $(N,\gamma)$-good points with a bound $C$, for fixed choices of $\left\{a_{n}\right\}$, $f$ and $C$, at a fixed time $N$ as:
\begin{equation*}
    G^{\gamma,C}_{N}=\left\{x\in X \mid \left\lvert A_{N}f(x) \right\rvert \leq C\cdot  N^{-\gamma} \cdot \lVert f \rVert_{Lip} \right\}.
\end{equation*}
We denote the complement of the set $G^{\gamma,C}_{N}$ by $B^{\gamma,C}_{N}$.
\end{defn}
We readily have the following inclusion $G^{\gamma,C}_{N}\subset G^{\gamma,C_1}_{N}$ for all $C_{1}\geq~C$.
Furthermore, fixing $\gamma,C$, for all fixed $0<\gamma'<\gamma$ there exists $N_0(\gamma,C,\gamma')$ such that for all $N>N_0$ we have $G^{\gamma',1}_{N}\subset G^{\gamma,C}_{N}$.

\begin{rem}
In the discussion that follows, we will fix $C,\gamma$ as constants which quantify the decay rate of the ergodic averages.
For each such specific choice, we will get a bound over the set of points for which (a sub-sequence of) the ergodic averages do not decay in this specific rate (c.f.~\eqref{eq:hdim-bound-gamma-C}).
Then we will take the exponent determining the decay rate, $\gamma$, to $0$. This will lead to a bound over the set of points for which the associated sequence of ergodic averages have a positive limit superior (c.f.~\eqref{eq:hdim-bound-C})).
Furthermore taking $C$ to infinity, and using a union bound, we deduce the final bound of~\eqref{eq:hdim-estimat-final}.
The constant $C'=~C'(C,f,\left\{a_n\right\})$ introduced in Corollary~\ref{cor:inheritance} is fixed once a choice of a constant $C$ and a choice of a Lipschitz function $f$ have been made and a polynomially bounded sampling sequence is given.
The constant $C_{\text{mix}}$ appearing throughout the proofs is fixed given choices of the function $f$ and the sampling sequence $\left\{a_{n}\right\}$, and does not tend to $0$ as the other parameters.
\end{rem}

From now on, assume that $f$ is a fixed Lipschitz function of vanishing integral for which the system $(X,u_{t})$ is polynomially mixing with a rate $\alpha$. We fix a sampling sequence $\left\{a_{n}\right\}$ which is polynomially bounded and a constant $C''$. 
Pick any $\gamma$ which is admissible with respect to Corollary~\ref{cor:estimate-good-set}, namely $\gamma<\alpha'$, where $\alpha'=\alpha'(\alpha)$ introduced in Lemma~\ref{lem:quant-alpha-mixing}.
We fix a number $C>0$ and we will consider the sets $B^{\gamma,C}_{N}, G^{\gamma,C}_{N}$ with respect to those fixed parameters.

\begin{cor}\label{cor:estimate-good-set}
Assume that $\left(X,u_{t}\right)$ is polynomial mixing with rate $\alpha$ for a Lipschitz function $f$ with vanishing integral. Then we have 
 
\begin{equation*}
    \mu\left(B^{\gamma,C}_{N}\right) \leq C_{\text{mix}}(f)\cdot C^{-2} \cdot  N^{2\gamma-2\alpha'}\cdot \lVert f\rVert^{2}_{Lip},
\end{equation*}
where $\alpha'=~\alpha'(\alpha)$ and $C_{\text{mix}}=C_{\text{mix}}(f)$ as in Lemma~\ref{lem:quant-alpha-mixing}.
\end{cor}
\begin{proof}
Using Chebyshev's inequality we get
\begin{equation*}
    \mu\left\{x\in X \mid \lvert A_{N}f(x) \rvert >C\cdot N^{-\gamma} \right\} \leq C^{-2}\cdot N^{2\gamma}\cdot \left\lVert A_{N}f(x)\right\rVert_{L^{2}}^{2}.
\end{equation*}
In view of the mixing estimates of Lemma~\ref{lem:quant-alpha-mixing} we deduce
\begin{equation*}
    \begin{split}
        \mu\left(B^{\gamma,C}_{N}\right) &\leq C_{\text{mix}}\cdot  C^{-2}\cdot N^{2\cdot\gamma-2\alpha'} \cdot \lVert f \rVert^{2}_{L^{2}}\\
        &\leq C_{\text{mix}}\cdot  C^{-2}\cdot N^{2\gamma-2\alpha'}\cdot\lVert f \rVert_{Lip}^{2}.
    \end{split}
\end{equation*}
\end{proof}

\begin{obs}\label{obs:geometrical-estimate}
Assume that $x$ is a $(N,\gamma)$-Good point with a bound $C$ for a \emph{Lipschitz function $f$} of vanishing integral.
There exists a number $C'=~C'(C,f,\left\{a_n\right\})\geq~C>~0$ such that for any point $y\in X$ satisfying $d_{X}(x,y)<\delta$ for any $\delta \leq N^{-2d-\gamma}$ the following estimate holds:
\begin{equation*}
    \lvert A_{N}f(y) \rvert \leq  C'\cdot N^{-\gamma}\cdot \lVert f \rVert_{Lip}.
\end{equation*}

\end{obs}
\begin{proof}[Proof of observation]
Write $y=hx$ for $h\in SL_{2}(\mathbb{R})$ where $h=\left(\begin{smallmatrix} a & b \\ c & d\end{smallmatrix} \right)$ with $|b|,|c|<\delta, |1-a|,|1-d|<\delta$.
We calculate the deviation of $u_{a_t}.y$ from $u_{a_t}.x$ as follows:
\begin{align}\label{eq:adjoint-computation}
\begin{split}
u_{a_t}.y &= u_{a_t}.hx \\
	&= \left(u_{a_t}hu_{-a_t}\right).\left(u_{a_t}.x\right) \\
	&= \left(\begin{smallmatrix} 1 & a_t \\ 0 & 1\end{smallmatrix} \right)\cdot \left(\begin{smallmatrix} a & b \\ c & d\end{smallmatrix} \right) \cdot \left(\begin{smallmatrix} 1 & -a_t \\ 0 & 1\end{smallmatrix} \right).\left(u_{a_t}.x\right) \\
	&= \left(\begin{smallmatrix} a+a_{t}\cdot c & b+a_{t}\cdot (d-a)-(a_t)^{2}c \\ c & d-a_{t}\cdot c\end{smallmatrix} \right).\left(u_{a_{t}}.x\right).
\end{split}
\end{align}
Choose $C''=C''(\left\{ a_n\right\})$ so that the maximal entry of this matrix for all $0\leq t \leq N$ is bounded in absolute value by $C''\cdot N^{2d}$. It is possible to choose such a number $C''$ independently of $N$ by the assumption over the polynomial growth of the sequence $\left\{a_n\right\}$.

As $\delta\leq N^{-2d-\gamma}$ and the function $f$ is a Lipschitz function, and $a_{t}$ grows polynomially with rate $d$, we have that $$\lvert f(u_{a_{t}}.y)-f(u_{a_{t}}.x)\rvert\leq C''\cdot N^{-\gamma}\cdot \lVert f \rVert_{Lip}$$ for all $1\leq t \leq N$, as $|b|,|c|,|1-a|,|1-d| \leq \delta$ and by the choice of $C''$.
Hence we may write
\begin{equation}\label{eq:nearby-pt-estimate}
    \begin{split}
    \lvert A_{N}f(y) \rvert &\leq \lvert A_{N}f(y) - A_{N}f(x) \rvert + \lvert A_{N}f(x) \rvert \\
    &\leq C''\cdot  N^{-\gamma}\cdot \lVert f \rVert_{Lip} + \lvert A_{N}f(x) \rvert.
\end{split}
\end{equation}
As $x\in G^{\gamma,C}_{N}$, we have that $\lvert A_{N}f(x)\rvert\leq~C\cdot N^{-\gamma}\cdot\lVert f\rVert_{Lip}$ and the required estimate follows with  $C'=C+C''$.
\end{proof}

We summarize the above observation in the following corollary:
\begin{cor}\label{cor:inheritance}
Given a sampling sequence $\left\{a_n \right\}$ which is polynomially bounded, a Lipschitz function $f$ with vanishing integral and a constant $C$, there exists $C'=C'(C,f,\left\{a_n\right\}) \geq C$ such that if $dist(y,G^{\gamma,C}_{N})<~N^{-(2d+\gamma)}$ then $y\in G^{\gamma,C'}_{N}$.
\end{cor}

We first describe the proof of Theorem~\ref{thm:hdim-estimate} for the case where $X=~G/\Gamma$ is \emph{compact}, afterwards we indicate the necessary changes to address the general case.

An immediate corollary to Observation~\ref{obs:geometrical-estimate} is the following isolation lemma:
\begin{lem}\label{lem:isolation-of-bad}
For any $y\in B^{\gamma,C'}_{N}$ and any $\delta\leq N^{-2d-\gamma}$ we have that $B_{\delta}(y) \cap G^{\gamma,C}_{N} = \emptyset$, where $C'=C'(C,f,\left\{a_n\right\})$ as in Corollary~\ref{cor:inheritance}.
\end{lem}
\begin{proof}
If not, picking $x\in B_{\delta}(y)\cap G^{\gamma,C}_{N}$ as an $(N,\gamma)$-good point with a bound $C$, we have that $y\in G^{\gamma,C'}_{N}$ in view of Corollary~\ref{cor:inheritance}.
\end{proof}
Let $N\left(B^{\gamma,C}_{N}, \delta\right)$ denote the largest cardinality of $\delta$-separated subset of $B^{\gamma,C}_{N}$
From the previous Lemma, one may deduce a bound on the cardinality of $delta$-separated points in $B^{\gamma,C'}_{N}$ in the following manner:
\begin{prop}\label{prop:seperated}
Fix $f,\left\{a_n\right\},\gamma$ and $C$. Consider $C'$ as in Corollary~\ref{cor:inheritance}. Then the quantity $N\left(B^{\gamma,C'}_{N}, 2\delta\right)$ satisfies the bound
\begin{equation*}
    N\left(B^{\gamma,C'}_{N},2 \delta\right) \ll_{X} C_{\text{mix}}\cdot C^{-2}\cdot \delta^{-3}\cdot N^{2\gamma-2\alpha'} \cdot \left\lVert f\right\rVert_{Lip}^{2}
\end{equation*}
for any $\delta\leq N^{-2d-\gamma}$.
\end{prop}
\begin{proof}
Let $S$ be any $2\delta$-separated subset of $B^{\gamma, C'}_{N}$.
The open balls $B_{\delta}(x)$ for $x\in S$ are pairwise disjoint. Hence we deduce the bound
\begin{equation*}
    \mu\left(\cup_{x\in S}B_{\delta}(x) \right) \gg_{X} \delta^{3}\cdot \lvert S \rvert,
\end{equation*}
as $\mu$ is a $3$-dimensional measure (as $\mu$ is the Haar measure of $X=G/\Gamma$ and $G=SL_{2}(\mathbb{R})$, so $\dim(G)=3$). We will suppress the explication of dependence of constants which are depending solely on $X$ from now on.
We also have that $\cup_{x\in S}B_{\delta}(x)$ is disjoint from $G^{\gamma,C}_{N}$ from Lemma~\ref{lem:isolation-of-bad}, hence 
\begin{equation*}
\begin{split}
    \mu\left(\cup_{x\in S}B_{\delta}(x)\right) &\leq 1-\mu\left(G^{\gamma,C}_{N}\right) \\
    &=\mu\left(B^{\gamma,C}_{N} \right) \\
    &\leq C_{\text{mix}}\cdot C^{-2}\cdot N^{2\gamma-2\alpha'}\cdot \lVert f\rVert_{Lip}^{2}.
\end{split}
\end{equation*}
by the measure estimate given in Corollary~\ref{cor:estimate-good-set}.
Therefore we get 
$$\lvert S \rvert \ll_{X} C_{\text{mix}}\cdot C^{-2}\cdot \delta^{-3}\cdot N^{2\gamma-2\alpha'}\cdot \lVert f\rVert_{Lip}^{2}.$$
\end{proof}

Fix some $\theta>0$ small, and define a sequence $\{N_{i}\}$ by $$N_{i}=\left[(1+\theta)^{i}\right].$$
We define the \emph{exceptional set} associated with $\left\{N_{i}\right\}$, denoted $S(\theta)$ is
\begin{equation*}
    S(\theta)=\left\{x\in X \mid \overline{\lim}_{i} \lvert A_{N_i}f(x) \rvert>0 \right\}.
\end{equation*}
We also define $S(\theta,C)=\left\{x\in X \mid \overline{\lim}_{i} \lvert A_{N_i}f(x) \rvert>C \right\}$ and we have $S(\theta)=\cup_{C>0}S(\theta,C)$.

Fix some $C>0$ and the associated $C'$, where $C'=C'(C,f,\left\{a_n\right\})$ as defined in Observation~\ref{obs:geometrical-estimate}.
For each $N_{i}$, we define $\delta_{i}=0.99\cdot N_{i}^{-2d-\gamma}$.

We estimate the series appearing in~\eqref{eq:dim-bound-seperation} (up to a minor change of $\delta$ to $2\delta$) using the bounds for separated sets as achieved in  Proposition~\ref{prop:seperated} with our choices of $\delta_i$:
\begin{equation}\label{eq:D-estimate}
    \begin{split}
    \sum_{i\geq 1}N\left(B_{N_i}^{\gamma,C'},2\delta_{i}\right)\cdot \left(4\cdot\delta_{i}\right)^{D} &\leq 4^{D}\cdot C_{\text{mix}}\cdot C^{-2} \cdot \sum_{i\geq 1}\delta_{i}^{-3}\cdot N_{i}^{2\gamma-2\alpha'}\cdot \delta_{i}^{D} \cdot \left\lVert f\right\rVert_{Lip}^{2} \\
    &\leq 4^{D}\cdot C_{\text{mix}} \cdot C^{-2}\cdot \lVert f\rVert_{Lip}^{2} \cdot \sum_{i\geq 1}N_{i}^{6d+5\gamma-2\alpha'-D(2d+\gamma)}.
    \end{split}
\end{equation}
If we constraint the number $D$ so that the exponent appearing in the series at ~\eqref{eq:D-estimate} is negative, the series will converge, as $N_{i}$ is a \emph{lacunary} sequence.
Therefore we require that 
\begin{equation*}
    6d+5\gamma-2\alpha'-D(2d+\gamma)<0,
\end{equation*}
or equivalently
\begin{equation*}
    D>3-\frac{2\alpha'-2\gamma}{2d+\gamma}.
\end{equation*}
In view of the definition of Hausdorff dimension and the discussion which follows Definition~\ref{defn:seperated-set}, the above estimate shows that
\begin{equation}\label{eq:hdim-bound-gamma-C}
    \dim_{H}(S(\theta,C')) \leq 3-\frac{2\alpha'-2\gamma}{2d+\gamma}.
\end{equation}\
Letting $\gamma$ approach $0$ along a sequence shows that any fixed $\theta$ and $C$, for $C'=C'(C,f,\left\{a_n\right\})$ as in Observation~\ref{obs:geometrical-estimate}, we get
\begin{equation}\label{eq:hdim-bound-C}
    \dim_{H}(S(\theta,C')) \leq 3-\frac{\alpha'}{d}.
\end{equation}
Now we let $C$ approach infinity along a sequence, the associated numbers $C'=C'(C,f,\left\{a_n\right\})$ tend to infinity as well, as can be seen by the choice of $C'$ in Observation~\ref{obs:geometrical-estimate}. Using a union bound for the Hausdorff dimensions of the sets $\left\{S(\theta,C')\right\}_{C'>0}$, we infer that for any fixed $\theta$ we have
\begin{equation*}
    \dim_{H}(S(\theta)) \leq 3-\frac{\alpha'}{d}.
\end{equation*}
Now we will let $\theta$ approach $0$ along a decreasing sequence, as by Bourgain's argument (as described in \S~\ref{sec:ergodic-theorem}), it is enough to show convergence of $A_{N}f(x)$ along the sub-sequences $A_{N_i}f(x)$.

Using the union bound for Hausdorff dimensions, by taking $\theta\to 0$ along a decreasing sequence, we see that
\begin{equation}\label{eq:hdim-estimat-final} 
\dim_{H}\left(\left\{x\in X\mid \overline{\lim}|A_{N}f(x)|>0 \right\}\right)\leq~3-~\frac{\alpha'}{d},
\end{equation}
finishing the proof of Theorem~\ref{thm:hdim-estimate} in the case of compact $X$.
When $X$ is not compact, we cover $X$ by countably many compact subsets $\left\{X_{i}\right\}_{i\in\mathbb{N}}$. For each such subset $X_{i}$ we construct a covering argument as indicated above, resulting in a bound of the form
\begin{equation*}
\dim_{H}\left(\left\{x\in X_{i}\mid \overline{\lim}|A_{N}f(x)|>0 \right\}\right)\leq~3-~\frac{\alpha'}{d}
\end{equation*}
and using the monotonicity property of the Hausdorff dimension - 
\begin{align*}
\dim_{H}\left(\left\{x\in X\mid \overline{\lim}|A_{N}f(x)|>0 \right\}\right) &= \dim_{H}\left(\cup_{i\in\mathbb{N}}\left\{x\in X_{i}\mid \overline{\lim}|A_{N}f(x)|>0 \right\}\right) \\ 
&\leq~3-~\frac{\alpha'}{d},
\end{align*} this concludes the proof of Theorem~\ref{thm:hdim-estimate} in the general case.

\begin{rem}
Examining the above estimates, one may see that we can recover a bound over the set of points for which the averages $\lvert A_{N}f(x) \rvert$ do not decay at a certain rate $C(f,\left\{a_n\right\})\cdot N^{-\gamma}$.
It is unknown to the author if there is any particular interest in those subsets (namely if the set of points with different convergence rates for the sparse averages, for the same function).
\end{rem}

Now we turn into proving Corollary~\ref{cor:decay-estimate}. We first show the following strengthening of Lemma~\ref{lem:quant-alpha-mixing} for the case of the average along polynomial trajectories.
Fix $p(x)\in~\mathbb{Z}[x]$ to be some non-constant polynomial of degree $d$, and define the operator 
\begin{equation*}
    A^{\text{poly}}_{N}f(x)=\frac{1}{N}\sum_{i=0}^{N-1}f\left(T^{p(i)}x\right).
\end{equation*}
\begin{lem}\label{lem:quant-mixing-horocycle}
Assume that $(X,T)$ is polynomially mixing for $f$ with rate $\alpha$, then for all $N$ large enough (depending on the polynomial $p$) we have that $$\| A^{\text{poly}}_{N}f\|_{L^{2}(X,\mu)} \leq C_{\text{mix}}\cdot N^{-\alpha''}\|f\|_{L^{2}(X,\mu)}$$ for $\alpha''=\frac{1}{2}\min\{1,d\cdot \alpha\}$, and $C_{\text{mix}}=C_{\text{mix}}(f,p)>0$ is some constant.
\end{lem}
\begin{proof}
The proof follows the computations demonstrated in Lemma~\ref{lem:quant-alpha-mixing}.
By similar computation as above, 
\begin{equation}\label{eq:poly-sample-mixing}
\| A^{poly}_{N}f\|_{L^{2}(X,\mu)}^{2} \leq \frac{\| f\|^{2}_{L^{2}(X,\mu)}}{N}+\frac{1}{N^2}\sum_{n,m=0,\ n>m}^{N}\left<T^{p(n)-p(m)}f,f\right>.
\end{equation}
As $p(n)$ is a polynomial of degree $d$, we have that $p(n)-p(n-k)\gg_{p}~k\cdot~n^{d-1}$ for any $1 \leq k<n$, as long as $n\geq O_{p}(1)$ for some constant $O_{p}(1)$ depending on the polynomial $p$. Using this estimate in the above expression yields - 
\begin{align*}
\| A^{\text{poly}}_{N}f\|_{L^{2}(X,\mu)}^{2} &\leq \frac{\| f\|^{2}_{L^{2}(X,\mu)}}{N}+\frac{2}{N^2}\sum_{n=1}^{N}\sum_{k=1}^{n-1}\left<T^{p(n)-p(n-k)}f,f\right> \\
&\leq \frac{\| f\|^{2}_{L^{2}(X,\mu)}}{N}+\frac{2}{N^2}\sum_{n= 1}^{O_{p}(1)}\sum_{k=1}^{n-1} \left<T^{p(n)-p(n-k)}f,f\right> \\
&\ + \frac{2}{N^2}\sum_{n= O_{p}(1)+1}^{N}\sum_{k=1}^{n-1} \left<T^{p(n)-p(n-k)}f,f\right>  \\ 
&\leq \frac{\| f\|^{2}_{L^{2}(X,\mu)}}{N}+\frac{2}{N^2}O_{p}(1)^2 \| f\|^{2}_{L^{2}(X,\mu)} \\
&\ +\frac{2}{N^2}\sum_{n=1}^{N}\sum_{k=1}^{n-1} C\cdot k^{-\alpha} n^{-\alpha\cdot(d-1)}\| f\|^{2}_{L^{2}(X,\mu)} \\ 
&\ll_{p} \frac{\| f\|^{2}_{L^{2}(X,\mu)}}{N}+\frac{C}{N^{2}}\sum_{n=1}^{N}n^{1-d\cdot\alpha}\| f\|^{2}_{L^{2}(X,\mu)} \\
&\ll_{p} 3\max\{1,C\}\cdot \left(N^{-1}+N^{-d\cdot\alpha}\right)\| f\|^{2}_{L^{2}(X,\mu)},
\end{align*}
for $C=C(f)$ as in Definition~\ref{defn:poly-mixing}.
\end{proof}

For the proof of Corollary~\ref{cor:decay-estimate}, we use an approximation argument combined with known estimates of decay of matrix coefficients.
\begin{proof}[Proof of Corollary~\ref{cor:decay-estimate}]
It is enough to deduce this estimate over a countable dense family of Lipschitz functions $f$, showing that for each such function $f$, the dimension of its exceptional set is bounded and concluding using the union bound for Hausdorff dimension.
Given $f$ a Lipschitz function of bounded support, we introduce a $K$-finite approximation $f^{N,\gamma}$ which satisfy $\|f-f^{N,\gamma}\|_{L^{\infty}} \leq N^{-\gamma}$ and $\lVert f^{N,\gamma} \rVert_{L^{2}} \leq \lVert f \rVert_{L^{2}}$.  
By examining the proof of Lemma~\ref{lem:lip-approx} we can obtain such a $K$-finite function $f^{N,\gamma}$ whose $K$-span subspace satisfies $\dim\left<K.f^{N,\gamma}\right>~\leq~ 2N^{\gamma+\epsilon}+~1$.

We have 
\begin{equation}\label{eq:lipschitz-rate}
\begin{split}
\mu\left\{x \in X \mid |A^{\text{poly}}_{N}f(x)| \geq \frac{2}{N^\gamma} \right\} & \leq \mu\left\{x \in X \mid \left\lvert A^{\text{poly}}_{N}f^{N,\gamma}(x)\right\rvert \geq \frac{1}{N^\gamma} \right\} \\
 & \leq C_{\text{mix}}\cdot N^{2\gamma}\left\|A^{\text{poly}}_{N}f^{N,\gamma}\right\|_{L^{2}(X,\mu)}^{2}.
\end{split}
\end{equation}
By using well-known bounds towards decay of matrix coefficients of \emph{K-finite vectors} cf. \cite[Equations (9.1),(9.4),(9.6)]{venkatesh2010sparse}, we have that
\begin{equation}
\left\lvert\left<u_{t}f,f\right>\right\rvert \leq \dim\left<K.f\right> \left(1+\lvert t\rvert\right)^{-\Re(s_1)}\|f\|_{L^{2}(m)}^{2}
\end{equation}
for every $K$-finite $f$, where $\lambda_{1}=s_{1}(1-s_{1})$ is the first non-trivial eigenvalue of the Laplacian of the locally symmetric space $K\backslash G /\Gamma$, considering $\Re\left(s_1\right)\leq 1/2$.
Hence in our notations of Lemma~\ref{lem:quant-mixing-horocycle}, we have a mixing rate estimate of
\begin{equation*}
    \alpha''=\min\left\{1/2,d\cdot\Re(s_1)\right\}
\end{equation*}
with a constant $C_{\text{mix}}=\sqrt{1+2\dim\left<K.f\right>}\leq\sqrt{3\dim\left<K.f\right>}$.
Therefore the estimate for $\mu\left\{x \in X \mid |A_{N}f^{N,\gamma}(x)| \geq \frac{C}{N^\gamma} \right\}$ given in Corollary~\ref{cor:estimate-good-set} takes the form of 

\begin{equation}\label{eq:smooth-conv-estimate}
\begin{split}
\mu\left\{x \in X \mid \left\lvert A^{\text{poly}}_{N}f^{N,\gamma}(x)\right\rvert \geq \frac{1}{N^\gamma} \right\} & \leq C^{-2}\cdot  N^{2\gamma}\cdot N^{-2\alpha''} \cdot 3\cdot \dim\left<K.f^{N,\gamma}\right>\cdot \|f^{N,\gamma}\|_{L^2(m)}^{2}\\
 & \leq 3\cdot C^{-2}\cdot N^{3\gamma+2\epsilon-2\alpha''}\|f\|_{L^2(X,\mu)}^{2}.
\end{split}
\end{equation}
Continuing verbatim as the proof of Theorem~\ref{thm:hdim-estimate}, we conclude that - 
\begin{equation}
\dim_{H}\left(\left\{x\in X\mid \overline{\lim}\lvert A^{\text{poly}}_{N}f(x)\rvert>0 \right\}\right) \leq 3-\frac{\alpha''}{d},
\end{equation}

where $\alpha''=(1/2)\min\{1,d\cdot\Re(s_1)\}$.
\end{proof}

Specifically, for the case of $SL_{2}(\mathbb{R})/SL_{2}(\mathbb{Z})$, and under the Selberg-Ramanujan conjecture for any homogeneous space which is given by a quotient with a principal congruence subgroup, we have the equality $\Re(s_{1})=1/2$. Using this data for the average along the squares, resulting in exceptional set of dimension less than $3-\frac{1}{4}=2.75$.
The best known bound today, due to Kim-Sarnak~\cite{kim2003refined}, amounts to $\Re(s_{1})\geq~1/2-7/64$. 
As for any $d>2$ this bound satisfies $d\cdot\Re{s_1} \geq 1$, we conclude that the bound for the exceptional set which amounts to the Selberg-Ramanujan bound is attained \emph{unconditionally} for polynomials of degree $3$ and higher. For quadratic polynomials, we have a bound of $3-\frac{25}{128}~=2.804\ldots$.

We now demonstrate a related result along the lines described above regarding averages over the primes in a unipotent orbit.
For $f\in L^{2}_{0}(X)$ a Lipschitz function, we denote 
\begin{equation*}
    P_{N}f(x_0)=\frac{1}{N}\sum_{n=1}^{N}f(u_{p_n}.x_0),
\end{equation*}
where $p_{n}$ denote the $n$'th prime number.
\begin{thm}
$\dim_{H}\left( \left\{x\in X \mid \overline{\lim} \lvert P_{N}f(x) \rvert >0\right\} \right) \leq 3-\sigma,$ where $\sigma=\frac{1}{2}\min\{1,\Re(s_1))\}$ for $s_1$ which satisfy $\lambda_{1}=s_{1}(1-s_{1})$, $\Re(s_{1})\leq 1/2$, where $\lambda_{1}$  is the first non-zero eigenvalue of the Laplacian on $X$.
\end{thm}
\begin{proof}
By explicit computation we have - 
\begin{align}\label{eq:primes-mixing}
\begin{split}
    \|P_{N}f(x)\|_{L^{2}(X)}^{2} &=\frac{1}{N^{2}}\sum_{n,m=1}^{N}\left<u_{p_{n}}.f,u_{p_{m}}.f\right> \\
    &\leq \frac{\|f\|_{L^{2}(X)}^{2}}{N} +\frac{2}{N^2}\sum_{k=1}^{2N\log(N)}d_{2N\log(N)}(k)\left<u_{k}f,f\right>,
    \end{split}
\end{align}
where $d_{M}(k)=\lvert\left\{p_{1},p_{2} \leq M, p_{1},p_{2} \text{ primes} \mid k=p_{2}-p_{1} \right\}\rvert$.
By the Bombieri-Davenport theorem (\cite[Theorem~$3.11$]{halberstam1974sieve}), we have the following bound for every $M\in\mathbb{N}$:
\begin{align*}
    d_{M}(k) &\leq 8\prod_{p>2}\left(1-\frac{1}{(p-1)^{2}} \right) \prod_{2<p\mid M}\frac{p-1}{p-2} \frac{M}{\log^{2}(M)}\left(1+O\left(\frac{\log\log(M)}{\log(M)}\right)\right) \\
    &\leq 16 \log\log(3M) \frac{M}{\log^{2}(M)}.
\end{align*}
Therefore we can estimate the second summand in \eqref{eq:primes-mixing} as -
\begin{align*}
    \frac{2}{N^2}\sum_{k=1}^{2N\log(N)}d_{2N\log(N)}(k)\left<u_{k}f,f\right> &\leq \frac{64\log\log(6N\log(N))\log(N)}{N\log^{2}(2N\log(N))}\sum_{k=1}^{2N\log(N)}C\cdot k^{-s}\|f\|_{L^{2}(X)}^{2} \\
    &\ll_{f} N^{-s+\epsilon},
\end{align*}
where the dependence over $f$ is by a factor of its $L^2$ norm and the mixing constant $C_{\text{mix}}$ as in Definition~\ref{defn:poly-mixing}.
Define $\alpha'''=\frac{1}{2}\min\{1,s-\epsilon\}$, then the above computation shows $$\|P_{N}f\|_{L^{2}(X)} \leq~ C_{\text{mix}} \cdot N^{-\alpha'''}\|f\|_{L^{2}(X)}.$$
We define the set of $(N,\gamma)$-Good points with a constant $C$, $G_{N}^{\gamma,C}$ to be all the points $x\in~ X$ for which $\lvert P_{N}f(x) \rvert \leq C\cdot  N^{-\gamma}\left\lVert f \right\rVert_{Lip}$.
We have $\mu\left(G^{\gamma,C}_{N}\right) \geq 1-C^{-2}\cdot C_{\text{mix}}\cdot N^{2\gamma+\epsilon-2\alpha'''}\cdot\lVert f \rVert_{Lip}^{2}$.
Considering Observation~\ref{obs:geometrical-estimate}, together with the fact that $p_{n}\ll_{\epsilon}~n^{1+\epsilon}$ for every $\epsilon>0$, we deduce that for any $C$, there exists $C'=C'(C,f)>0$ such that if $dist(y,G^{\gamma,C}_{N})<N^{-2-2\epsilon-\gamma}$ then $y\in G^{\gamma,C'}_{N}$.
Using the argument as described above yields the following bound for the Hausdorff dimension 
\begin{equation*}
    \dim_{H}\left( \left\{x\in X \mid \overline{\lim} \lvert P_{N}f(x) \rvert >0\right\} \right) \leq 3-\alpha'''\leq 3-\frac{s}{2}+\epsilon.
\end{equation*}
As the above bound holds for any $\epsilon>0$, letting $\epsilon$ tend to zero along a decreasing sequence, we conclude the theorem.
\end{proof}

\section{Removing the dependence on the spectral gap}\label{sec:s-gap}
In this section we address the issue of achieving spectral-gap free estimates for the Hausdorff dimension of the exceptional set.
Unlike the previous sections, the results we present in this section are specialized for $G=SL_{2}$.
In the next section we show how to deduce a suitable theorem for general groups based upon this result. \

We begin by introducing some required background about representation theory of $SL_{2}(\mathbb{R})$ and $PGL_{2}(\mathbb{R})$.

\subsection{Background from representation theory and construction of a model}
The well-known classification of unitary irreducible representations of $G=SL_{2}(\mathbb{R})$ and $G=PGL_{2}(\mathbb{R})$ (\cite[Section $2.5$]{knapp2001representation}, \cite[Theorem $2.6.7$]{bump1997automorphic}) asserts that any such representation $\left(\rho,\mathcal{H}\right)$  belongs to one of the following classes:
\begin{enumerate}
	\item The trivial representation.
	\item Discrete series representation.
	\item Limits of discrete series representation.
	\item Principal series representation.
	\item Complementary series representation.
\end{enumerate}
Here the principal series representations are representations which are formed by parabolic induction of a unitary character of the Cartan subgroup, and complementary series representations are formed by parabolic induction of a non-unitary character of the Cartan subgroup.
The representations are indexed by their Casimir eigenvalue.
Amongst the representations, the only ones which admit a spherical vector are the trivial one, principal and complementary series representations.

We say that a representation $\left(\rho,\mathcal{H}\right)$ is $G$-\emph{tempered} if it is weakly contained in the left regular representation $L^{2}(G)$.
In practice in order to conclude whether a representation is tempered or not, it is enough to consider the integrability of matrix coefficients of $K$-finite vectors (c.f. \cite[Theorem $8.53$]{knapp2001representation}).
By the asymptotic bounds for matrix coefficients (\cite[Theorem $8.47.b$]{knapp2001representation} and \cite[equation (9.6)]{venkatesh2010sparse}), we conclude that the $G$-\emph{tempered} representations are the discrete series, limits of discrete series and principal series representations.
Hence for a tempered representation $\left(\rho,\mathcal{H}\right)$, the decay of its matrix coefficients is majorized by the Harish-Chandra bound.

For a lattice $\Gamma \leq G$ we view the space $L^{2}_{0}(G/\Gamma)$ as a $G$-representation space, we decompose the space as follows - $$L^{2}_{0}(G/\Gamma)=V_{\text{tempered}} \oplus V_{\text{non-tempered}},$$ where $V_{\text{tempered}}$ consists of all the $G$-tempered representations which are weakly-contained in $L^{2}_{0}(G/\Gamma)$, and $V_{\text{non-tempered}}$ is the ortho-complement space.
It is known that $V_{\text{non-tempered}}$ consists of finitely many summands, each of which is isomorphic to a complementary series representation with spectral parameter $s\in (0,1/2)$ where we write the Laplacian eigenvalue as $\lambda=s(1-s)$.
All the representations which occur in $V_{\text{tempered}}$ have a spectral parameter $s$ satisfying $\Re(s)=1/2$.
In view of the results from the previous section, for any smooth $f\in V_{\text{tempered}}$ we have a uniform bound on the Hausdorff dimension of the exceptional set.
So from now on, we will assume we are given a Lipschitz function $f \in V_{\text{non-tempered}}$.

The \emph{line model} (\cite[$\S 3.1$, Equation~$9$]{gelfand1968representation}) for principal series representation with spectral parameter $s$ is the space $L^{2}(\mathbb{R},\text{Lebesgue})$, equipped with the action - 
\begin{equation}\label{eq:line-model-action}
\begin{pmatrix}
	a & b \\
	c & d
\end{pmatrix} . f(x) = \lvert -cx+a \rvert^{-2s} f\left(\frac{dx-b}{-cx+a} \right).
\end{equation}
In order to define a line model for the complementary series representation with spectral parameter $s$, we define the following bi-linear form on the space $L^{1}_{Loc}(\mathbb{R})$ - 
\begin{equation}\label{eq:inner-product-complementary}
\langle f_{1},f_{2} \rangle = \int_{x=-\infty}^{\infty} f_{1}(x) \int_{y=-\infty}^{\infty}\overline{f_{2}(y)}\lvert x-y \rvert^{2s-2}dydx.
\end{equation}
While it is not evident from the definition of this bilinear form, this is an hermitian form (\cite[$\S 3.2$, Equation~$3$]{gelfand1968representation}, and completing it to a Hilbert space gives rise to an irreducible unitary $G$-representation isomorphic to the complementary series representation of parameter $s$ equipped with the action given in \eqref{eq:line-model-action}.

A closely related model to the line model is the Kirillov model, given by doing a Fourier transform on the function space.
The resulting inner product in the complementary series representation case is given by - 
\begin{equation}\label{eq:inner-product-complementary-Kirillov}
\langle f_{1},f_{2} \rangle = \int_{-\infty}^{\infty} \hat{f_{1}}(t)\overline{\hat{f_{2}}(t)}\lvert t \rvert^{1-2s}dt.
\end{equation}

As the $U$-action in the line model is given by translation, in the Kirillov model the action will be given by a multiplier, namely  
\begin{equation*}
u_{n}.\hat{f}(t)=e^{2\pi i n t}\hat{f}(t).
\end{equation*}

We caution the reader that our spectral parameter $s$ is different from the one taken in \cite{gelfand1968representation}, as we write the Laplacian eigenvalues differently.

\subsection{Bounding the exceptional set via estimates on exponential sums and oscillatory integrals}
In the previous section, the method to bound the dimension of the exceptional set relied on two components - effective mixing estimate and a geometric estimate arising from polynomial divergence.
As the mixing rate depends on the specific lattice and its spectral gap, we will change the operator whose norm we are bounding.
For a given non-constant polynomial $p(x)\in\mathbb{Z}[x]$ of degree $d$, we define the following operator -
\begin{equation*}
    B_{N}f(x)=\frac{1}{N}\sum_{n=0}^{N-1}f(u_{p(n)}.x)-\frac{1}{N}\int_{0}^{N}f(u_{p(n)}.x)dn
\end{equation*} for a function $f\in L^{2}_{0}(G/\Gamma)$.
For the \emph{continuous time} average, equidistribution theorems applies (c.f. Lemma~\ref{lem:dani-smillie-by-parts}, and \cite[Corollary $1.1$]{shah1994limit}) for every \emph{non $U$-periodic} point $x\in G/\Gamma$ we have that 
\begin{equation*}
\frac{1}{N}\int_{0}^{N}f(u_{p(n)}.x)dn \to 0,
\end{equation*}
as $N$ tends to infinity.
Moreover, one can get a quantitative version of the equidistribution theorem which will decay in a rate related to the spectral mixing rate (cf. \cite{burger1990horocycle} Theorem $2$, \cite{strombergsson2013deviation} Theorem $1$).
One may view the operator $B_{N}$ defined above as an operator which bounds a second term rate for the discrete average, assuming that the main term of the discrete average decays like the continuous average.
We will show an estimate related to Corollary~\ref{cor:estimate-good-set} for the operator $B_{N}$ which is independent of the spectral gap $s_{0}(\Gamma)$.
As the geometrical estimate given in Observation~\ref{obs:geometrical-estimate} works also for comparing the integral averages of two nearby points $\frac{1}{N}\int_{0}^{N}f(u_{p(n)}.x)dn, \frac{1}{N}\int_{0}^{N}f(u_{p(n)}.y)dn$, the rest of the packing argument will work exactly the same.

We remark here that by the explicit spectral resolution of $L^{2}(G/\Gamma)$ obtained by Selberg, every homogeneous space $X=G/\Gamma$ has a spectral gap $s_{0}=s_{0}(\Gamma)>0$. This spectral parameter $s_{0}$ will appear in two forms in the proceeding discussion - first - as a parameter which controls the decay rate of certain functions appearing in complementary series representations which are weakly contained in $L^{2}(X)$ and second - as a normalizing factor for those functions. As we are interested in polynomial estimates for $\|B_{N}f\|_{L^{2}(X)}$, we will show that we can let $s_{0}$ tend to $0$ in the decay rate, while keeping the normalizing factor constant (as this constant will disappear in the computation of the actual Hausdorff dimension, and for a given lattice $\Gamma$, this constant is a fixed non-zero number).

We begin with several auxiliary lemmas regarding estimates of exponential sums, oscillatory integrals and Fourier transforms of basis vectors of a given $SL_{2}$-representation $V_{s}$ from the complementary series.
\begin{lem}\label{lem:mean-value}
Fix some $\alpha>0$. For every $N>0$ and $\lvert t \rvert \leq N^{-((d-1)+\alpha)}$ we have
\begin{equation*}
\left\lvert \frac{1}{N}\sum_{n=0}^{N-1}e^{2\pi i p(n) t}-\frac{1}{N}\int_{0}^{N}e^{2\pi i p(n) t}dn \right\rvert \leq O\left(N^{-\alpha}\right).
\end{equation*} 
\end{lem}
\begin{proof} 
We have the following bound - 
\begin{equation*}
\left\lvert e^{2\pi i p(n)t}-\int_{h=0}^{1}e^{2\pi i p(n+h)t}dh \right\rvert~\leq \max_{h\in \left[0,1\right]}\left\lvert e^{2\pi i p(n)t}-e^{2\pi i p(n+h)t} \right\rvert,
\end{equation*}
using the mean-value theorem we can estimate this difference by \\  $4\pi \lvert t\rvert\cdot\max_{h\in \left[0,1\right]}\lvert p'(n+h)\rvert$.
Rewriting the expression above as 
\begin{equation*}
\left|\frac{1}{N}\sum_{n=0}^{N-1}\left(e^{2\pi i p(n)t}-\int_{h=0}^{1}e^{2\pi i p(n+h)t}dh\right) \right|,
\end{equation*}
we deduce the following bound - 
\begin{align*}
\left|\frac{1}{N}\sum_{n=0}^{N-1}\left(e^{2\pi i p(n)t}-\int_{h=0}^{1}e^{2\pi i p(n+h)t}dh \right)\right| &\leq \frac{4\pi \lvert t\rvert }{N}\sum_{n=0}^{N-1}\max_{h\in \left[0,1\right]}\lvert p'(n+h)\rvert \\ 
&= 4\pi \lvert t \rvert \cdot O_{p}\left(N^{d-1}\right).
\end{align*}
Choosing $t$ accordingly, we get the desired bound.
\end{proof}

\begin{lem}\label{lem:stationary-phase}
For every non-zero $t$ we have $\frac{1}{N}\int_{0}^{N} e^{2\pi i p(n)t}dn = O_{p}\left(\frac{1}{N|t|^{1/d}}\right)$.
\end{lem}
\begin{proof}
Fixing $N$, from the fact that $p^{(d)}(n)\equiv const \neq 0$ and the van-der-Corput lemma for the method of stationary-phase (c.f. \cite[Lemma $8.2$]{iwaniec2004analytic}) we deduce that $\int_{0}^{N} e^{2\pi i p(n)t}dn = O_{p}\left(\lvert t\rvert^{-1/d}\right)$, hence the claim follows.
\end{proof}

\begin{cor}
For any $\lvert t \rvert \geq N^{-((d-1)+\alpha)}$ we have $\frac{1}{N}\int_{0}^{N} e^{2\pi i p(n)t}dn = O\left(N^{-\frac{1-\alpha}{d}}\right)$.
\end{cor}

We would need the following estimates regarding the Fourier transform of basis vectors in the line model of complementary series - 
\begin{prop}
    Let $f_0(x)=\frac{1}{\left(x^2+1 \right)^{s}}$ be the spherical basis vector for a given $SL_{2}$-representation of spectral parameter $0<s<\frac{1}{2}$.
    Then we have the following expression for its Fourier transform -  $$\hat{f_0}(t)=\frac{\pi^{1/2}}{\Gamma(s)2^{s-1/2}}\cdot t^{s-1/2}K_{s-1/2}(t),$$ where $K_{v}(t)$ stands for the modified Bessel function which defined as
    \begin{equation*} 
    K_{v}(t)=~C(I_{v}(t)+~I_{-v}(t)).
    \end{equation*}
\end{prop}
\begin{proof}
    By the formula for Fourier transform we have - 
    \begin{equation*} \hat{f_0}(t)=\int_{-\infty}^{\infty} \frac{e^{2\pi i xt}}{\left(x^2+1 \right)^{s}}dx=2\int_{x=0}^{\infty}\frac{\cos(xt)}{(x^2+1)^{s}}dx. \end{equation*}
    Using Basset's integral representation of the modified Bessel function (c.f. \cite[$\S6.16$, Equation~$1$]{watson1995treatise})
    \begin{equation}\label{eq:basset-integral} 
    K_{\nu}(t) = \frac{\Gamma\left(\nu+\frac{1}{2}\right)2^{\nu}}{\pi^{\frac{1}{2}}t^{\nu}}\int_{x=0}^{\infty}\frac{\cos(xt)}{(x^2+1)^{\nu+\frac{1}{2}}}dx,
    \end{equation}
    we get $\hat{f_0}(t)=\frac{\pi^{1/2}}{\Gamma(s)2^{s-1/2}}\cdot t^{s-1/2}K_{s-1/2}(t)$.
\end{proof}
Using the explicit series expansion of the Bessel function \cite[$\S3.7$, Equations $2,6$]{watson1995treatise} we have the following estimate - 
\begin{lem}[Asymptotics of spherical function near $0$]\label{lem:bessel-estimate}
    For $0<t < 1$, 
    \begin{equation*}
        \lvert \hat{f_0}(t) \rvert \lesssim_{s} t^{2s-1}.
    \end{equation*}
\end{lem}

Furthermore, using the asymptotics for the Bessel function near $\infty$ (c.f. \cite[$7.23$, Equation $1$]{watson1995treatise} we infer the following - 
\begin{lem}[Asymptotics of spherical function near $\infty$]\label{lem:bessel-estimate-infty}
For $t\gg 1$, 
\begin{equation*}
    \lvert \hat{f}_{0}(t) \rvert \ll t^{s-1}e^{-t}.
\end{equation*}
\end{lem}

Moreover, we obtain similar estimates for any other basis vector as well, informally such a result follows from the fact that the decay rate at infinity of all the other basis vectors are the same as the spherical vectors, hence the asymptotics near zero of their Fourier transforms behave the same -
\begin{lem}\label{lem:extended-bessel-estimate}
Let $n$ be an integer and $f_{n}$ be the weight $n$ vector in the complementary series representation $V_{s}$, for $0<t<1$ we have $\lvert \hat{f_{n}}(t) \rvert \lesssim t^{2s-1}$.
\end{lem}
\begin{proof}

The formula for vector of weight $n$ in $V_{s}$ is given by - 
\begin{equation}\label{eq:basis-vector}
    f_{n}(x)=\left(\frac{x-i}{x+i}\right)^{n}\cdot\frac{1}{\left(x^2+1\right)^{s}}.
\end{equation}
Rewriting this expression as either $\frac{\left(x-i\right)^{2n}}{\left(x^2+1\right)^{n+s}}$ or $\frac{\left(x+i\right)^{2n}}{\left(x^2+1\right)^{n+s}}$, according to the sign of $n$, we compute the following formula for the Fourier transform, based on Basset's integral representation 
\begin{equation}\label{eq:fourier-trans-basis-vector}
    \hat{f_{n}}(t)=\frac{\pi^{1/2}}{\Gamma(2n+s)2^{2n+s-1/2}}\left(\frac{d}{dt}\pm i\right)^{2n}\cdot t^{n+s-1/2}K_{n+s-1/2}(t).
\end{equation}
By the connection formulas for the derivative of the Bessel function (\cite[$\S3.71$, Equations $5,6$]{watson1995treatise}) we have 
\begin{equation*}
    \left(\frac{1}{t}\frac{d}{dt}\right)^{m}t^{v}K_{v}(t)=(-1)^{m}t^{v-m}K_{v-m}(t), 
\end{equation*}
therefore we have 
\begin{equation*}
    \frac{d}{dt}P(t)t^{v}K_{v}(t) = P'(t)t^{v}K_{v}(t)-P(t)t^{v}K_{v-1}(t)
\end{equation*}
for any polynomial $P(t)$. 
Using the asymptotics of the Bessel functions $K_{v}(t)$ and the explicit expression for the Fourier transform \eqref{eq:fourier-trans-basis-vector} we see that $$t^{n+s-1/2}K_{n+s-1/2-2n}(t) \lesssim_{s}~t^{2s-1}$$ and the other terms in the expression are of powers higher or equal to $2s-1$, hence the claim follows.
\end{proof}

Using the expansion derived in the proof of the previous lemma, we obtain the following estimate
\begin{lem}\label{lem:extended-bessel-estimate-infty}
For $t\gg 1$,
\begin{equation*}
    \lvert \hat{f_{n}}(t) \rvert \ll_{n} t^{\lvert n \rvert+s-1}e^{-t}.
\end{equation*}
\end{lem}
We note here that although the dependence in $n$ is inherent, we would need to consider only $K$-finite functions in the proof, namely functions for which $\hat{f}_{n}\equiv 0$ for $\lvert n \rvert \gg 0$, leading to uniformity of the estimates in the $n$-aspect. 

We would also need some results regarding moments of exponential sums.
\begin{defn}
We say that a polynomial $p\in\mathbb{Z}[x]$ of degree $d$ has \emph{$q$-moment cancellation with level $\ell$} if the following inequality holds for all $\epsilon>0$,
\begin{equation}
    \int_{0}^{1} \left\lvert \frac{\sum_{n=0}^{N-1}e^{2\pi i p(n)t}}{N} \right\rvert^{q}dt \ll_{\epsilon} N^{-(\ell-\epsilon)}.
\end{equation}
\end{defn}
The following lemma follows immediately from the definition.
\begin{lem}[Moment bound]\label{lem:hua} Let $p\in\mathbb{Z}[x]$ be a polynomial which has $q$-moment cancellation with level $\ell$, then for every $\beta>0,\epsilon>0$ we have - 
\begin{equation*}
\int_{0}^{N^\beta} \left\lvert\frac{\sum_{n=0}^{N-1}e^{2\pi i p(n)t}}{N} \right\rvert^{q}dt \ll_{\epsilon} N^{-\ell+\beta+\epsilon}.
\end{equation*}
\end{lem}
We would be interested in small values of $q$ for which a given polynomial $p$ has $q$-moment cancellation with level strictly larger than $\deg(p)-1$.

Hua's bound~\cite{hua1938waring} shows that any integer  polynomial of degree $d$ has $2^{d}$-moment cancellation with level $d$.
Recently Bourgain\cite[Theorem~$10$]{bourgain2016vinogradov} improved upon Hua's bound and obtained that any monomial $p(x)=~x^{d}$ has $d(d+1)$-moment cancellation with level $d$.

Now we are ready to present the main estimate for this subsection.
Consider the set of $K$-finite functions, as in Definition~\ref{defn:K-finite}. 
Those functions can be thought of ``generalized trigonometric polynomials''.
In a unitary $G$-representation, the set of $K$-finite vectors form a dense subspace, as a consequence of the Peter-Weyl theorem (c.f. \cite[Theorem~$1.12$]{knapp2001representation}).
We refer the reader to~\cite[Section \S8.2]{knapp2001representation} for more background about the role of $K$-finite functions in representation theory of semisimple Lie groups.
\begin{thm}\label{thm:mixing-est-s-gap-free}
Assume that $f$ is a K-finite function which belongs to a complementary series representation of parameter $s$ occurring in $L^{2}_{0}(G/\Gamma)$ with $\|f\|_{L^{2}(G/\Gamma)}=1$, then we have $\|B_{N}f\|_{L^{2}(G/\Gamma)} \ll_{f} N^{-\delta}$ for some $\delta>0$ which is independent of $s$. In particular one may take any $\delta<\frac{1}{2q+1}$ where the polynomial $p(x)$ has $2q$-moment cancellation with level $d$.
\end{thm}

\begin{proof}[Proof of Theorem~\ref{thm:mixing-est-s-gap-free}]
As $f$ is $K$-finite, we can write $f$ in a $K$-spherical Fourier decomposition 
\begin{equation*}
    f(x)=\sum_{n=-L}^{L}a_{n}f_{n}(x)
\end{equation*} for some fixed $L>0$.

Using the Kirillov model for complementary series representation, we have the following expression for $\|B_{N}f\|_{L^{2}(G/\Gamma)}$
\begin{equation}\label{eq:integral-before-uncertainty}
\|B_{N}f\|^{2}_{L^{2}(G/\Gamma)} = \int_{-\infty}^{\infty} \left\lvert \frac{1}{N}\sum_{n=0}^{N-1}e^{2\pi i p(n) t}-\frac{1}{N}\int_{0}^{N}e^{2\pi i p(n) t}dn \right\rvert^{2} \lvert \hat{f}\rvert^{2} \lvert t \rvert^{1-2s}dt.
\end{equation}
Let $\beta>0$ be a positive number to be determined later.
We split the integral as follows:

\begin{align}\label{eq:integral-after-uncertainty}
\begin{split}
\|B_{N}f\|^{2}_{L^{2}(G/\Gamma)} & = \int_{-N^{\beta}}^{N^{\beta}} \left\lvert \frac{1}{N}\sum_{n=0}^{N-1}e^{2\pi i p(n) t}-\frac{1}{N}\int_{0}^{N}e^{2\pi i p(n) t}dn \right\rvert^{2} \lvert \hat{f}\rvert^{2} \lvert t \rvert^{1-2s}dt \\ 
& + \int_{\lvert t \rvert \geq N^{\beta}}\left\lvert \frac{1}{N}\sum_{n=0}^{N-1}e^{2\pi i p(n) t}-\frac{1}{N}\int_{0}^{N}e^{2\pi i p(n) t}dn \right\rvert^{2} \lvert \hat{f}\rvert^{2} \lvert t \rvert^{1-2s}dt.
\end{split}
\end{align}

We further refine the dissection of the first term:
\begin{flalign}
\begin{split}
&\int_{-N^{\beta}}^{N^{\beta}} \left\lvert \frac{1}{N}\sum_{n=0}^{N-1}e^{2\pi i p(n) t}-\frac{1}{N}\int_{0}^{N}e^{2\pi i p(n) t}dn \right\rvert^{2} \lvert \hat{f}\rvert^{2} \lvert t \rvert^{1-2s}dt \\
&=\int_{\lvert t \lvert \leq N^{-((d-1)+\alpha)}}\left\lvert \frac{1}{N}\sum_{n=0}^{N-1}e^{2\pi i p(n) t}-\frac{1}{N}\int_{0}^{N}e^{2\pi i p(n) t}dn \right\rvert^{2} \lvert \hat{f}\rvert^{2} \lvert t \rvert^{1-2s}dt \\
&+\int_{N^{-((d-1)+\alpha)}\leq \lvert t \rvert \leq N^{\beta}} \left\lvert\frac{1}{N}\sum_{n=0}^{N-1}e^{2\pi i p(n) t}-\frac{1}{N}\int_{0}^{N}e^{2\pi i p(n) t}dn \right\rvert^{2} \lvert \hat{f}\rvert^{2} \lvert t \rvert^{1-2s}dt.
\end{split}
\end{flalign}

For the first summand, using Lemma~\ref{lem:mean-value} we get 
\begin{equation}\label{eq:integral-first-first}
\int_{\lvert t \rvert \leq N^{-((d-1)+\alpha)}}\left\lvert \frac{1}{N}\sum_{n=0}^{N-1}e^{2\pi i p(n) t}-\frac{1}{N}\int_{0}^{N}e^{2\pi i p(n) t}dn \right\rvert^{2} \lvert \hat{f}\rvert^{2} \lvert t \rvert^{1-2s}dt \leq N^{-2\alpha}\|f\|^{2}_{L^{2}(m)}.
\end{equation}
For the second summand, we first use a trivial bound

\begin{align*}
&\int_{N^{-((d-1)+\alpha)}\leq \lvert t \rvert \leq N^{\beta}} \left\lvert \frac{1}{N}\sum_{n=0}^{N-1}e^{2\pi i p(n) t}-\frac{1}{N}\int_{0}^{N}e^{2\pi i p(n) t}dn \right\rvert^{2} \lvert \hat{f}\rvert^{2} \lvert t \rvert^{1-2s}dt \\
&\leq 2\int_{N^{-((d-1)+\alpha)}\leq \lvert t \rvert \leq N^{\beta}} 
\left\lvert \frac{1}{N}\sum_{n=0}^{N-1}e^{2\pi i p(n) t}\right\rvert^2\lvert \hat{f}\rvert^{2} \lvert t \rvert^{1-2s}dt \\ 
&+ 2\int_{N^{-((d-1)+\alpha)}\leq \lvert t \rvert \leq N^{\beta}} \left\lvert \frac{1}{N}\int_{0}^{N}e^{2\pi i p(n) t}dn\right\rvert^2\lvert\hat{f}\rvert^{2} \lvert t \rvert^{1-2s}dt.
\end{align*}
Using Lemma~\ref{lem:stationary-phase} for the later summand, we infer
\begin{align}\label{eq:integral-first-second-second}
\begin{split}
&\int_{N^{-((d-1)+\alpha)}\leq \lvert t \rvert \leq N^{\beta}} \left\lvert \frac{1}{N}\sum_{n=0}^{N-1}e^{2\pi i p(n) t}-\frac{1}{N}\int_{0}^{N}e^{2\pi i p(n) t}dn \right\rvert^{2} \lvert \hat{f}\rvert^{2} \lvert t \rvert^{1-2s}dt  \\
&\ll_{\gamma,\beta}\int_{N^{-((d-1)+\alpha)}\leq \lvert t \rvert \leq N^{\beta}} \left\lvert \frac{1}{N}\sum_{n=0}^{N-1}e^{2\pi i p(n) t}\right\rvert^{2} \lvert \hat{f}\rvert^{2} \lvert t \rvert^{1-2s}dt +O\left(N^{-\frac{1-\alpha}{d}}\right) \|f\|^{2}_{L^{2}(m)}.
\end{split}
\end{align}
Choose $q>1$ such that $p(x)$ has $2q$-moment cancellation with level $d$, and let $q'$ be its Holder conjugate.
Using Holder's inequality

\begin{align}
\begin{split}
    \int_{N^{-((d-1)+\alpha)}\leq \lvert t \rvert \leq N^{\beta}} \left\lvert \frac{1}{N}\sum_{n=0}^{N-1}e^{2\pi i p(n) t}\right\rvert^{2} \lvert \hat{f}\rvert^{2} \lvert t \rvert^{1-2s}dt \\
    \leq \left(\int_{N^{-((d-1)+\alpha)}\leq \lvert t \rvert \leq N^\beta} \left\lvert \frac{1}{N}\sum_{n=0}^{N-1}e^{2\pi i p(n) t}\right\rvert^{2q}dt\right)^{1/q} &\left(\int_{N^{-((d-1)+\alpha)}\leq \lvert t \rvert \leq N^{\beta}} \lvert \hat{f}\rvert^{2q'} \lvert t \rvert^{q'(1-2s)}dt\right)^{1/q'}.
\end{split}
\end{align}

By explicit computation using Lemma~\ref{lem:bessel-estimate} and Lemma~\ref{lem:extended-bessel-estimate} we have for $f=f_{n}$ a given basis vector, for $0<t<1$ 

\begin{equation*}
    \lvert \hat{f}(t) \rvert^{2q'} \ll_{s} \lvert t\rvert^{2q'(2s-1)},
\end{equation*}
and for $t>1$
\begin{equation*}
    \lvert \hat{f}(t) \rvert^{2q'} \ll_{s} \lvert t\rvert^{2q'(L+s-1)}e^{-2q't},
\end{equation*}
therefore

\begin{flalign*}
\int_{N^{-((d-1)+\alpha)}\leq \lvert t \rvert \leq N^{\beta}} \lvert \hat{f}\rvert^{2q'} \lvert t \rvert^{q'(1-2s)}dt 
& \ll_{f} \int_{N^{-((d-1)+\alpha)}\leq \lvert t \rvert \leq 1} \lvert t^{2s-1}\rvert^{2q'} \lvert t \rvert^{q'(1-2s)}dt \\
&+ \int_{1\leq \lvert t \rvert \leq N^{\beta}} \lvert t^{L+s-1}e^{-t}\rvert^{2q'} \lvert t \rvert^{q'(1-2s)}dt \\ 
& = \int_{N^{-((d-1)+\alpha)}\leq \lvert t \rvert \leq 1}t^{q'(2s-1)}dt + \int_{1\leq \lvert t \rvert \leq N^{\beta}}t^{2q'L+q}e^{-2q't}dt \\
& \leq N^{((d-1)+\alpha)(q'(1-2s)-1)}+\left(2q\right)^{-(2q'L+q'+1)}\Gamma(2q'L+q'+1).
\end{flalign*}
Assuming that $N^\beta$ is much larger than $L$ 
the above computations lead to a bound of the form

\begin{align}
\begin{split}
\left(\int_{N^{-((d-1)+\alpha)}\leq \lvert t \rvert \leq N^{\beta}} \lvert \hat{f}\rvert^{2q'} \lvert t \rvert^{q'(1-2s)}dt\right)^{1/q'} 
&\ll N^{\left((d-1)+\alpha\right)(q'-1)/q'}\\
&=N^{\left((d-1)+\alpha\right)/q}.
\end{split}
\end{align}
Combining the above with Lemma~\ref{lem:hua} we can conclude the following bound:
\begin{equation}\label{eq:integral-first-second}
\begin{split}
\int_{N^{-((d-1)+\alpha)}\leq \lvert t \rvert \leq N^{\beta}} \left\lvert \frac{1}{N}\sum_{n=0}^{N-1}e^{2\pi i p(n) t}\right\rvert^{2} \lvert \hat{f}\rvert^{2} \lvert t \rvert^{(1-2s)}dt &\ll
N^{-\frac{d-\epsilon-\beta}{q}+\frac{(d-1)+\alpha}{q}} \\
&= N^{-\frac{1-\alpha-\epsilon}{q}}.
\end{split}
\end{equation}

For the remaining integral in~\eqref{eq:integral-after-uncertainty},
we use the asymptotics derived in Lemma~\ref{lem:bessel-estimate-infty} and Lemma~\ref{lem:extended-bessel-estimate-infty},
and the assumption we are given basis vectors with weights in $\left[-L,L\right]$ to deduce - 
\begin{align}\label{eq:integral-second}
\begin{split}
&\int_{\lvert t\rvert \geq N^{\beta}}\left\lvert \frac{1}{N}\sum_{n=0}^{N-1}e^{2\pi i p(n)t}-\frac{1}{N} \int_{0}^{N}e^{2\pi i p(n)t}dn \right\rvert^{2}\lvert \hat{f}(t)\rvert^{2}\lvert t \rvert^{1-2s}dt 
\\
&\leq 2\int_{\lvert t\rvert \geq N^{\beta}}\lvert \hat{f}(t)\rvert^{2}\lvert t \rvert^{1-2s}dt \\
&\leq 2\int_{\lvert t\rvert \geq N^{\beta}}\lvert \sum_{i=-L}^{L} \alpha_{i}\hat{f}_{i}\rvert^{2}\lvert t \rvert^{1-2s}dt \\
&\leq 4\int_{\lvert t\rvert \geq N^{\beta}}\sum_{i=-L}^{L}\lvert\alpha_{i}\rvert^{2}\lvert \hat{f}_{i}(t)\rvert^{2}\lvert t \rvert^{1-2s}dt \\
&\leq 4\int_{\lvert t\rvert \geq N^{\beta}}\sum_{i=-L}^{L}\lvert\alpha_{i}\rvert^{2}\lvert t\rvert^{2(L+s-1)} \lvert t \rvert^{1-2s}dt \\
&= 4\int_{\lvert t\rvert \geq N^{\beta}}\lvert t\rvert^{2L-1}e^{-2\lvert t\rvert}dt \cdot \|f\|_{L^2(m)}^{2}\\
&=2^{2L}\Gamma\left(2L-1,2N^{\beta}\right)\cdot \|f\|_{L^2(m)}^{2},
\end{split}
\end{align}
where $\Gamma(s,x)$ stands for the incomplete Gamma function $$\Gamma(s,x) = \int_{x}^{\infty}t^{s-1}e^{-t}dt.$$ 
Using the asymptotics $\Gamma(s,x)\approx x^{s-1}e^{-x}$ for $x\to \infty$, we deduce that the integral is bounded by $2^{2L}(2N^\beta)^{2L-1}e^{-N^{\beta}}\cdot \|f\|_{L^2(m)}^{2} \leq N^{4L\beta}e^{-N^{\beta}}\cdot~\|f\|_{L^2(m)}^{2}$. \\

Combining \eqref{eq:integral-after-uncertainty},\eqref{eq:integral-first-first},\eqref{eq:integral-first-second} and \eqref{eq:integral-second} we conclude
\begin{align}\label{eq:final-s-free-estimate}
\begin{split}
\|B_{N}(f)\|^{2}_{L^{2}(m)}/\|f\|_{L^2(m)}^{2} &\ll N^{-2\alpha}\\ &+\left(N^{-\frac{1-\alpha}{d}}+N^{-\frac{1-\alpha-\epsilon}{q}}\right) \\
&+ N^{4\beta L} e^{-N^\beta}.
\end{split}
\end{align}

As $L$ is fixed, choosing $\alpha<1$ and $\beta,\varepsilon$ small enough, we deduce the estimate.
One possible choice for the parameters is to take $\alpha\approx 1/(2q+1)$, and $\beta,\varepsilon\approx 0$ in order to get any $\delta<\frac{1}{2q+1}$ in Theorem~\ref{thm:mixing-est-s-gap-free}. In particular in view of Hua's bound, we may take any $\delta<\frac{1}{2^{d}+1}$. 
\end{proof}

\begin{rem}
In the case of monomials $p(x)=x^{d}$ of degree of larger than $6$, Bourgain's improvement is superior to Hua's bound, and shows it is enough to consider $q=d(d+1)/2$ in the above proof, leading to a bound of the form $N^{-\frac{1-\alpha-\epsilon}{d(d+1)/2}}$ in~\eqref{eq:integral-first-second}, resulting in an estimate of $\frac{1}{d(d+1)+1}$ for $\delta$. 

Moreover, the above proof shows that it is enough to consider $q$ such that $p(x)$ has $2q$-moment cancellation with level strictly larger than $d-1$, and not necessarily level $d$, although such $q$ will reflect on $\alpha$.
\end{rem}

Using Theorem~\ref{thm:mixing-est-s-gap-free}, the following corollary follows from Chebychev's inequality just like in Corollary~\ref{cor:estimate-good-set} -
\begin{cor}
Let $f\in L^2_{0}(G/\Gamma)$ be a K-finite function with $\|f\|_{L^2}=~1$ then the set $B^{\gamma,C}_{N}$ which is the complement of the set of $(N,\gamma)$-Good points with constant $C$ in $X$ satisfy $$\mu\left(B^{\gamma,C}_{N}\right) \ll_{f,p} N^{2\gamma+2\epsilon-2\delta}$$ for any $\epsilon>0$, where the implied constant depends on $f$ and the polynomial $p$.
\end{cor}
Using Observation~\ref{obs:geometrical-estimate} (note that the same Lipschitz estimate for pointwise comparison of the averages along the orbits holds for the integral as well), we use Proposition~\ref{prop:seperated} (with the minor modification of changing $2\cdot C$ to $4\cdot C$) and continue to construct a packing argument as described in $\S3$ in order to have the following estimate - 
\begin{equation}\label{eq:approximation-dimension-estimate-2}
    \dim_{H}\left\{x\in X \mid \overline{\lim}\left\lvert \frac{1}{N}\sum_{n=0}^{N-1}f\left(u_{p(n)}.x\right) - \frac{1}{N}\int_{0}^{N}f(u_{p(t)}.x)dt \right\rvert > 0 \right\} \leq 3-\frac{\delta}{d},
\end{equation}
for any $\delta<\frac{1}{2q+1}$, where $p$ has $2q$-moments cancellation of level $d$, for $K$-finite function $f$.

Now assuming that $f$ is smooth, we have rapid decay of the $K$-Fourier coefficients and $\|f-~f^{\left[-L,L\right]}\|_{\infty}\to~ 0$ as $L$ tends to $\infty$.
Therefore arguing similarly to \eqref{eq:smooth-conv-estimate}, given $x\in X$ so that $\overline{\lim} \lvert B_{N}f(x) \rvert >\varepsilon$, we choose $L$ large enough so that $\|f-~f^{\left[-L,L\right]}\|_{\infty} \leq \varepsilon/3$. 
Noticing that 
\begin{equation*}
    \sup_{x\in X}\lvert B_{N}(f(x)-f^{\left[-L,L\right]}(x)) \rvert \leq 2\|f-f^{\left[-L,L\right]}\|_{\infty},
\end{equation*}
we have for that $x\in X$ that $\overline{\lim}\left\lvert B_{N}f^{\left[-L,L\right]} \right\rvert \geq \varepsilon/3$, therefore
\begin{equation}
    \mu\left\{x\in X \mid \left\lvert B_{N}f(x) \right\rvert\geq \varepsilon \right\} \leq \mu\left\{x\in X \mid \left\lvert B_{N}f^{\left[-L,L\right]}(x) \right\rvert\geq \varepsilon/3 \right\} \ll_{f,p} N^{2\epsilon-2\delta}.
\end{equation}

Continuing in similar fashion to the proof of Corollary~\ref{cor:decay-estimate} we have the following estimate 
\begin{equation}\label{eq:approximation-dimension-estimate}
    \dim_{H}\left\{x\in X \mid \overline{\lim}\left\lvert \frac{1}{N}\sum_{n=0}^{N-1}f\left(u_{p(n)}.x\right) - \frac{1}{N}\int_{0}^{N}f(u_{p(t)}.x)dt \right\rvert > 0 \right\} \leq 3-\frac{\delta}{d},
\end{equation}
for any $\delta$ which is admissible for the estimate proven in Theorem~\ref{thm:mixing-est-s-gap-free}. In particular when $p$ is a monomial $p(x)=x^{d}$ we may take $\delta<~\max\left\{\frac{1}{2^{d}+1},\frac{1}{d(d+1)+1}\right\}$.

\begin{lem}\label{lem:dani-smillie-by-parts}
Assume that $f$ is a smooth function on $X$ with $\int_{X}fd\mu~=~0$, then for any non-$U$-periodic point $x_0 \in X$ we have $\lim_{T\to 0}\frac{1}{T}\int_{0}^{T}f(u_{p(t)}.x_0)dt~=~0$.
\end{lem}
\begin{proof}
To ease notation, we assume that the leading coefficient of $p(t)$ equals to $1$, the computation involved in the general case is similar.

Let $\varepsilon>0$ be given.
Denote $F(t)=\int_{0}^{t}f(u_{t}.x_0)dt$.

By the equidistribution theorem of Dani-Smillie~\cite{dani1993limit} for the continuous orbit at linear time, given $\varepsilon>0$, there exists $T_e=T_e(\varepsilon,f,x_0)$ such that for all $t\geq T_e$, we have $\left\lvert F(t)/t \right\rvert <\varepsilon$.
Without loss of generality, we may assume $T_e>1$.
As the polynomial $p(t)$ is fixed, there exists $T_{1}=T_{1}(p)$ such that for all $t>T_1$ we have :
\begin{equation*}
\begin{split}
0.99t^{d} &\leq p(t) \leq 1.01t^{d},\\
0.99dt^{d-1} &\leq p'(t) \leq 1.01dt^{d-1},\\
0.99d(d-1)t^{d-2} &\leq p''(t) \leq 1.01d(d-1)t^{d-2}.
\end{split}
\end{equation*}
We may assume without loss of generality that $T_{1}>1$, so in particular the function $p(t)$ is monotonically increasing for $t>T_1$.

Let $0<\delta<1$ be a small number to be determined later.
We assume that $T$ is large enough so that $\delta \cdot T>\max\{T_{1},T_{e}\}$.
We may write 
$$ \frac{1}{T}\int_{0}^{T}f\left(u_{p\left(t\right)}.x_{0}\right)dt=\frac{1}{T}\int_{0}^{\delta\cdot T}f\left(u_{p(t)}.x_0\right)dt+\frac{1}{T}\int_{\delta\cdot T}^{T}f\left(u_{p\left(t\right)}.x_{0}\right)dt $$
and so we get 
$$ \frac{1}{T}\int_{0}^{T}f\left(u_{p\left(t\right)}.x_{0}\right)dt=O_{f}\left(\delta\right)+\frac{1}{T}\int_{\delta\cdot T}^{T}f\left(u_{p\left(t\right)}.x_{0}\right)dt. $$
Substituting variables $s=p\left(t\right)$ lead to the following integral:
$$\frac{1}{T}\int_{\delta\cdot T}^{T}f\left(u_{p\left(t\right)}.x_{0}\right)dt=\frac{1}{T}\int_{p\left(\delta\cdot T\right)}^{p\left(T\right)}f\left(u_{s}.x_{0}\right)q\left(s\right)ds, $$
where $q(s)=dp^{-1}(s)/ds$.

Integration by parts give
\begin{equation}
    \begin{split}
        \frac{1}{T}\int_{p\left(\delta\cdot T\right)}^{p\left(T\right)}f\left(u_{s}.x_{0}\right)q\left(s\right)ds &= \frac{1}{T}\left[F\left(p\left(T\right)\right)\cdot q\left(p\left(T\right)\right)-F\left(p\left(\delta\cdot T\right)\right)\cdot q\left(p\left(\delta\cdot T\right)\right)\right]\\ 
        &\ -\frac{1}{T}\int_{p\left(\delta\cdot T\right)}^{p\left(T\right)}F\left(s\right)\cdot q'\left(s\right)ds.
    \end{split}
\end{equation}

In the region where we integrate, we have that $q(p(t)) = O_{p}\left(\left(\delta\cdot T\right)^{-(d-1)}\right).$
Considering the first summand, by the equidistribution theorem we get
$$ \frac{1}{T}\left\lvert F(p(t))\cdot q(p(t)) \right\rvert = O_{p}\left(\varepsilon\cdot \delta^{-(d-1)} \right). $$

For the second summand, applying the mean value theorem we get
$$ \frac{1}{T}\int_{p\left(\delta\cdot T\right)}^{p\left(T\right)}F\left(s\right)\cdot q'\left(s\right)ds= \frac{p\left(T\right)-p\left(\delta\cdot T\right)}{T}\cdot F\left(p\left(t\right)\right)\cdot q'\left(p\left(t\right)\right)$$  
for some $\delta\cdot T<t<T$.

Using the inverse function theorem, we may differentiate $q$ as follows
\begin{equation*}
    q'(p(t))=-\frac{p''(t)}{\left(p'(t)\right)^3}.
\end{equation*}
As $t>\delta\cdot~T>~T_1$ we deduce that $$q'\left(p\left(t\right)\right)=O_{p}\left(\frac{t^{d-2}}{t^{3\left(d-1\right)}}\right)=O_{p}\left(\left(\delta\cdot T \right)^{-2d+1}\right).$$ 
Hence we get $$F\left(p\left(t\right)\right)\cdot q'\left(t\right)=O_{p}\left(\varepsilon\cdot T^{-d+1} \cdot \delta^{-2d+1}\right).$$
Moreover, we have that $$\frac{p(T)-p(\delta\cdot T)}{T} = O_{p}\left(T^{d-1}\right).$$
So for the second summand we have the estimate
\begin{equation*}
\begin{split}
    \frac{p\left(T\right)-p\left(\delta\cdot T\right)}{T}\cdot F\left(p\left(t\right)\right)\cdot q'\left(p\left(t\right)\right) &= O_{p}\left(T^{d-1}\right) \cdot O_{p}\left(\varepsilon \cdot T^{-d+1}\cdot \delta^{-2d+1} \right) \\
    &= O_{p}\left(\varepsilon\cdot \delta^{-2d+1}\right).
\end{split}
\end{equation*}

Collecting the bounds we get
\begin{equation*}
    \frac{1}{T}\int_{0}^{T}f(u_{p(t)}.x_0)dt = O_{f}\left(\delta\right)+O_{p}\left(\varepsilon\cdot \delta^{-(d-1)}\right) + O_{p}\left(\varepsilon \cdot \delta^{-2d+1} \right),
\end{equation*}
for all $T$ such that $\delta\cdot T>\max\{T_{e},T_{1} \}$.
Hence by choosing $\delta=\varepsilon^{1/2d}$, we get 
$$ \frac{1}{T}\int_{0}^{T}f\left(u_{p(t)}.x_0\right)dt = O_{f,p}\left(\varepsilon^{1/2d}\right), $$
for all $T>\varepsilon^{-1/2d}\cdot\max\{T_e, T_1\}$.
\end{proof}

By the previous Lemma, we see that \emph{except for $U$-periodic points}, $\frac{1}{N}\int_{0}^{N}f(u_{p(t)}.x)dt\to~0$ as $N\to\infty$, hence by the inequality -
\begin{equation*}
    \begin{split}
        \left\lvert \frac{1}{N}\sum_{n=0}^{N-1}f\left(u_{p(n)}.x\right) \right\rvert &\leq \left\lvert \frac{1}{N}\sum_{n=0}^{N-1}f\left(u_{p(n)}.x\right) - \frac{1}{N}\int_{0}^{N}f(u_{p(t)}.x)dt \right\rvert\\
        &+ \left\lvert \frac{1}{N}\int_{0}^{N}f(u_{p(t)}.x)dt \right\rvert,
    \end{split}
\end{equation*}
we deduce the following theorem:
\begin{thm}\label{thm:s-gap-free-estimate}
\begin{equation}\label{eq:s-gap-free-estimate}
\begin{split}
    \dim_{H}\left(\left\{x\in Y \mid \overline{\lim}\left\lvert \frac{1}{N}\sum_{n=0}^{N-1}f\left(u_{p(n)}.x\right)\right\rvert > 0 \right\}\right) &= \dim_{H}\left(\left\{x\in Y \mid \overline{\lim}\left\lvert B_{N}f(x)\right\rvert > 0\right\} \right) \\ &\leq 3-\frac{\delta}{d},
\end{split}
\end{equation}
where $Y\subset X$ stands for the set of $U$-generic points, for all $\delta$ which are admissible for the estimate of Theorem~\ref{thm:mixing-est-s-gap-free}.
\end{thm}
We may write the exceptional set $E$ as $E=(E\cap Y)~\cup~ (E\cap ~Y^{c})$, where
$Y^{c}$ consists of the non-$U$-generic points (if those even exists, namely the lattice $\Gamma$ is non-uniform), which in the case of $SL_{2}$ are $U$-periodic points which form a finite ensemble of \emph{two-dimensional ''tubes''}, hence $\dim_{H}(E\cap~Y^{c}) \leq ~\dim_{H}(Y^{c})=~2$.
For a point $x\in E\cap Y$, we have that $\frac{1}{N}\int_{0}^{N}f(u_{p(t)}.x)dt \to~0$, hence we can deduce the following bound $\dim_{H}(E\cap Y)\leq 3-\frac{\delta}{d}$, using a union bound for the Hausdorff dimension, we deduce the following bound for the Hausdorff dimension of the whole exceptional set 
$$\dim_{H}(E) \leq \max\left\{3-\frac{\delta}{d},2\right\}.$$
In particular, for $p(n)=n^2$, we have $3-\frac{1}{10}=2.9$ as an upper bound for the exceptional set of convergence of square averages,  for any lattice $\Gamma \leq SL_{2}(\mathbb{R})$.

\section{Bounding exceptional sets for unipotent flows in general Lie groups}\label{sec:higher-dimensions}
In this section we show how to get a general bound for the dimension of the exceptional set of one-parameter unipotent flows in general Lie groups, based on the results of the prior sections.

Let $G$ be a real semi-simple linear Lie group, let $\Gamma \subset G$ be a \emph{lattice} in $G$, and let $U=\left\{u_{t} \right\}$ be a one-parameter unipotent group where $u_{t}=exp(t\cdot N)$ for some nilpotent element $N \in Lie(G)$.

Let $p \in \mathbb{Z}[x]$ be a polynomial of degree $d$, and for every continuous function $f$ with compact support on $G/\Gamma$, denote by $A^{\text{poly}}_{N}$ the following averaging operator - 
\begin{equation*}
    A^{\text{poly}}_{N}f(x_0)=\frac{1}{N}\sum_{n=0}^{N-1}f\left(u_{p(n)}.x_0\right).
\end{equation*}

By the Jacobson-Morozov theorem (see \cite{knapp2013lie}, Theorem~$10.3$), we can complete $N$ into a $\mathfrak{sl_{2}}$-triplet, denote the subgroup generated by this $\mathfrak{sl_{2}}$-triplet under the exponential map by $L$ which satisfies either $L\simeq~SL_{2}(\mathbb{R})$ or $L\simeq~PGL_{2}(\mathbb{R})$.

Let $(\rho,\mathcal{H})$ be a unitary representation of $G$. We can restrict $\rho$ to $L$ and get a unitary representation of $L$, $\rho\mid_{L}$. As $L$ is semi-simple, we can write $\mathcal{H}$ as a direct integral over irreducible unitary $L$-representations $\{V_{s}\}$ as follows - $\mathcal{H}=\int^{\oplus}V_{s}d\mu(s)$ for a suitable spectral measure $\mu$.
We have the following Parseval-type formula for a vector $v\in \mathcal{H}$ - 
\begin{equation*}
    \|v\|_{\mathcal{H}}^{2}=\int \|\pi_{s}(v)\|_{V_s}^{2}d\mu(s),
\end{equation*}
where $\pi_{s}:\mathcal{H}\to V_{s}$ is the associated projection operator.
As each $V_{s}$ is $L$-invariant, we can decompose the operator $A_{N}$ across the irreducible $L$-constitutes we have
\begin{equation}\label{eq:general-parsavel}
    \|A^{\text{poly}}_{N}v\|_{\mathcal{H}}^{2} = \int \|A^{\text{poly}}_{N}\pi_{s}(v)\|_{V_s}^{2} dm_{L}(s),
\end{equation}
where $dm_{L}$ is the associated spectral measure.
Using Lemma~\ref{lem:quant-mixing-horocycle}, we have that for each constituent in the decomposition and a $K$-finite vector $v$ - 
\begin{equation*}
    \|A^{\text{poly}}_{N}\pi_{s}(v)\|_{V_s}^{2} \leq N^{-2s'}\|\pi_{s}(v)\|_{V_s}^{2},
\end{equation*}
for some $s'$ which satisfy $s' \leq \min\left\{\frac{1}{2},(d-1)s_{1}\right\}$, where $s_{1}$ is the spectral gap of $V_{s}$ as $L$-representation, such a bound follows from the Harish-Chandra bound (\cite[Lemma $9.1$]{venkatesh2010sparse}).
Hence 
\begin{equation}
\begin{split}
    \|A^{\text{poly}}_{N}v\|_{\mathcal{H}}^{2} &\leq \int N^{-2s'}\|\pi_{s}(v)\|_{V_s}^{2}dm_{L}(s) \\
    &\leq N^{-2s'}\int\|\pi_{s}(v)\|_{V_s}^{2}dm_{L}(s) \\
    &= N^{-2s'}\|v\|_{\mathcal{H}}^{2}.
\end{split}
\end{equation}

We say that $u_{t}$ is of \emph{degree $\ell$} if every polynomial entry in variable $t$ of the matrix $u_{t}$ is a polynomial of degree less or equal to $\ell$  and $\ell$ is the minimal natural number with that property. 

A computation analogous to the computation done in Observation~\ref{obs:geometrical-estimate}, shows that for every matrix $h$, the entries of $ad_{u_{t}}(h)$ are polynomials of degree at-most $2\cdot \ell$, hence we can conclude the following observation.
\begin{obs}
Assume that $x,y \in G/\Gamma$  such that $x$ is a $(N,\gamma)$-Good point with a constant $C$, and $d(x,y) \leq N^{-2\ell\cdot d-\gamma}$ ,then $y$ is $(N,\gamma)$-Good point with a constant $C'$ for $C'=C'(C,f,p)$ as in Observation~\ref{obs:geometrical-estimate}.
\end{obs}
Continuing in a similar manner to the one described in $\S\ref{sec:h-dim}$, one concludes an upper bound for the dimension of the exceptional set of form $\dim(G)-\frac{s}{d\cdot\ell}$, where $s=s(G,\Gamma)$ is related to the spectral gap of the representation of $G$ on $L^{2}_{0}(G/\Gamma)$.

Now, building upon the results of $\S\ref{sec:s-gap}$, we are going to remove the dependence on the spectral gap.
We say that a point $x\in X$ is \emph{$U$-generic} for a one-parameter subgroup $U=\{u_{t}\} \leq G$ if for every $f\in~C_{c}(X)$ we have $\lim_{t\to\infty}\frac{1}{T}\int_{0}^{T}f(u_{t}.x)dt = \int_{X}fd\mu(X)$ where $\mu$ is the unique probability measure on $X$ induced from the Haar measure on $G$.
Define $I_{N}$ to be the following operator: $$I_{N}f(x)=\frac{1}{N}\int_{t=0}^{N}f(u_{p(t)}.x)dt.$$
We clearly have the following inequality
\begin{equation*}
    \overline{\lim} \|A^{\text{poly}}_{N}f\|_{L^{2}(m)} \leq \overline{\lim} \|A^{\text{poly}}_{N}f-I_{N}f\|_{L^{2}(m)} + \overline{\lim} \|I_{N}f\|_{L^{2}(m)}.
\end{equation*}
Similar to Lemma~\ref{lem:dani-smillie-by-parts}, we conclude that $I_{N}f(x)$ tend to $0$ as $N$ tends to infinity, for every $U$-generic point $x$.
\begin{lem}\label{lem:ratner-countable}
The non-$U$-generic points $x\in G/\Gamma$ are contained inside a countable union of varieties of co-dimension $1$, and in-particular \begin{equation*}
\dim_{H} \left\{x\in G/\Gamma \mid x \text{ is not }U\text{-generic} \right\} \leq \dim(G)-1.
\end{equation*}
\end{lem}
\begin{proof}
    Denote by $\mathcal{H}$ the collection of all closed connected subgroups $H$ of $G$ such that $H\cap \Gamma$ is a lattice in $H$ and the subgroup $S$ generated by all the unipotent one-parameter subgroups of $G$ contained in $H$ acts ergodically on $H\Gamma / \Gamma$ with respect to the $H$-invariant probability measure (i.e. $S=\text{Stab}_{G}(\mu)$ for some homogeneous $u_{t}$-invariant and ergodic measure $\mu$).
    By \cite[Corollary $A.(2)$]{ratner1991raghunathan} or  \cite[Proposition $2.1$]{dani1993limit}, there exists only countably many subgroups $H$ in $\mathcal{H}$.

    Let $W$ be a subgroup of $G$ generated by one-parameter unipotent subgroups of $G$ which are contained inside $W$.
    For $H\in \mathcal{H}$ we define the following sets:
    \begin{align*}
        N(H,W) &= \left\{g\in G\mid W\subset gHg^{-1} \right\},\\
        S(H,W) &= \bigcup_{H'\in\mathcal{H}, H'\subset H, H'\neq H} N(H',W),
    \end{align*}
    and we define the ''tube'' with respect to a subgroup $W$ as follows:
    \begin{equation*}
        T_{H}(W)=\pi(N(H,W))\setminus\pi(S(H,W)),
    \end{equation*}
    where $\pi:G\to G/\Gamma$ is the natural projection map.
    For any $g\Gamma \in G/\Gamma$ such that $g\Gamma \in T_{H}(W)$, the subgroup $gHg^{-1}$ equals to the stability group of the homogeneous measure which is supported on the orbit $\overline{W.(g\Gamma)}\subset G/\Gamma$.
    Specializing to $W=U$, the set of \emph{$U$-generic} points in $G/\Gamma$ is equal to $T_{G}(U)$.
    The singular set composed of \emph{non-$U$-generic} points is contained in  $\pi(S(G,U))$, which can be written as a countable union of sets of the form $\pi(N(H,U))$ for every $H\in \mathcal{H}$ different than $G$.
    The set $N(H,U)$ for any fixed $H\in\mathcal{H}$ is an \emph{analytic variety} and in particular, it is of dimension smaller than $\dim(G)$ unless it is equal to $G$.
    As the projection map from the group to the homogeneous space $\pi:G\to G/\Gamma$ is $1$-Lipschitz map, the Hausdorff dimensions of the various embeddings $\pi(N(H,U))$ cannot increase, thus bounding the Hausdorff dimension of each such ``tube'' $\pi(N(H,U))$.
    Therefore, the non-$U$-generic points are contained in a countable union of lower-dimensional varieties, and by the union property of the Hausdorff dimension, we deduce the result.
\end{proof}
Moreover, as we can compute $\|A^{\text{poly}}_{N}\pi_{V_{s}}f-I_{N}\pi_{V_{s}}f\|$ separately for every $L$-representation $V_{s}$ which is contained inside $L^{2}_{0}(G/\Gamma)$, we are essentially in the settings of Theorem~\ref{thm:mixing-est-s-gap-free}, with a sampling along a polynomial of degree $\ell$ from the unipotent flow in $L$, hence we have the following estimate - $\|A_{N}f-I_{N}f\|_{V_s} \leq~ N^{-c}\|f\|_{V_{s}}$ where $c=c(p)$ independent of $s$.
Using \eqref{eq:general-parsavel} we have - 
\begin{equation*}
    \|A^{\text{poly}}_{N}f-I_{N}f\|_{L^{2}(m)} \leq N^{-c}\|f\|_{L^{2}(m)},
\end{equation*}
for every $K$-finite function $f$ of vanishing integral.
We will use the following standard lemma:
\begin{lem}
Consider $C^{\infty}_{c}\left(X\right)$ as a $G$-representation. There exists a separable family of $K$-finite functions in $C^{\infty}_{c}\left(X\right)$. 
\end{lem}
For a proof, one may consult \cite[Section~\S6]{harish-chandra1966}.

We may represent the set of points $x\in G/\Gamma$ such that the samples $\left\{u_{p(n)}.x\right\}$ do not equidistribute as a countable union of exceptional sets with respect to functions in a separable family which can be obtained from the lemma.

Continuing in an analogues manner to Section $\S\ref{sec:s-gap}$ and the proof of  Theorem~\ref{thm:s-gap-free-estimate}, we deduce that
\begin{equation*}
    \dim_{H}\left\{x\in G/\Gamma \mid \left\{u_{p(n)}.x\right\}_{n=1}^{\infty} \text{ is not equidistributed} \right\} \leq \max\left\{\dim(G)-\frac{c}{d\cdot\ell},D_{T}\right\},
\end{equation*}
where $D_{T}=\dim_{H}(\pi\left(S\left(G,U\right)\right))$ is the dimension of the tube of the non-$U$-generic points, and in-particular $D_{T}\leq \dim(G)-1$.
Moreover, as $\ell$ is bounded (depending on $G$), we can make this estimate uniform over the one-parameter unipotent subgroups of $G$.

\bibliographystyle{plain}
\bibliography{bib-squares}

\end{document}